\documentclass[a4paper,11pt,twoside]{amsart}

\usepackage{amsfonts}
\usepackage{amsmath}
\usepackage{amsthm}
\usepackage{amssymb}
\usepackage[all]{xy}
\usepackage{hyperref}
\usepackage{aliascnt}

\newtheorem{thmInt}{Theorem}[section]

\newaliascnt{propInt}{thmInt}

\aliascntresetthe{propInt}
\newaliascnt{corInt}{thmInt}
\newtheorem{corInt}[corInt]{Corollary}
\aliascntresetthe{corInt}

\newaliascnt{prop}{thm}
\newtheorem{prop}[prop]{Proposition}
\aliascntresetthe{prop}

\newaliascnt{lem}{thm}
\newtheorem{lem}[lem]{Lemma}
\aliascntresetthe{lem}

\newaliascnt{cor}{thm}
\newtheorem{cor}[cor]{Corollary}
\aliascntresetthe{cor}

\newaliascnt{conjecture}{thm}

\aliascntresetthe{conjecture}

\newaliascnt{qn}{thm}

\aliascntresetthe{qn}

\theoremstyle{definition}

\newaliascnt{definition}{thm}
\newtheorem{definition}[definition]{Definition}
\aliascntresetthe{definition}

\newaliascnt{remark}{thm}
\newtheorem{remark}[remark]{Remark}
\aliascntresetthe{remark}

\newaliascnt{ex}{thm}
\newtheorem{ex}[ex]{Example}
\aliascntresetthe{ex}

\numberwithin{equation}{section}

\newcommand{\NN}{\mathbb{N}}
\newcommand{\ZZ}{\mathbb{Z}}

\newcommand{\iso}{\cong}

\newcommand{\farg}{-} 
\newcommand{\id}{\mathrm{id}}
\newcommand{\Id}{\mathrm{Id}}
\newcommand{\st}{\mid} 
\newcommand{\comp}{\circ} 
\newcommand{\mor}[1]{\xrightarrow{#1}}
\newcommand{\mono}{\hookrightarrow} 
\newcommand{\epi}{\twoheadrightarrow} 
\newcommand{\isomor}{\mor{\sim}} 
\newcommand{\card}[1]{\lvert#1\rvert} 
\newcommand{\rest}[1]{|_{#1}} 
\newcommand{\K}{\Bbbk} 

\newcommand{\Iso}{\mathrm{Isom}}

\newcommand{\cat}[1]{{\mathbf{#1}}} 

\newcommand{\Ob}[1]{\mathrm{Ob}(#1)} 
\newcommand{\fun}[1]{\mathsf{#1}} 
\newcommand{\Ho}[1]{\mathrm{Ho}(#1)} 
\newcommand{\Inf}[1]{\mathrm{Ho}_\infty(#1)} 
\newcommand{\red}[1]{\overline{#1}} 
\newcommand{\aug}[1]{#1^+} 
\newcommand{\B}{\mathrm{B}} 
\newcommand{\Bi}{\B_{\infty}} 
\newcommand{\coB}{\Omega} 
\newcommand{\adjpair}[4]{#1\colon#2\rightleftarrows#3:\!#4}
\newcommand{\sh}[2][1]{#2[#1]} 
\newcommand{\Tc}[1]{\mathrm{T}^c(#1)} 
\newcommand{\Tar}[1]{\red{\mathrm{T}}(#1)} 
\newcommand{\Tcr}[1]{\red{\mathrm{T}}^c(#1)} 
\newcommand{\Vid}[1][]{J_{#1}} 
\newcommand{\Vnid}[1][]{I_{#1}} 
\newcommand{\Vide}[1][]{J'_{#1}} 
\newcommand{\htp}{\sim} 
\newcommand{\htpm}{\asymp} 
\newcommand{\hiso}{\approx} 
\newcommand{\quot}[1]{/\!#1} 
\newcommand{\m}[2][]{\mathrm{m}^{#2}_{#1}} 
\newcommand{\lotimes}{\otimes^{\mathbb{L}}} 
\newcommand{\IHom}{\mathbb{R}\underline{Hom}} 
\newcommand{\filt}[1]{\mathrm{F}^{#1}} 
\newcommand{\gr}[1]{\mathrm{gr}^{#1}} 
\newcommand{\grVn}[1]{L_{#1}} 
\newcommand{\grVnid}[1]{I_{#1}} 
\newcommand{\htpiso}{homotopy isomorphism}
\newcommand{\htpequiv}{homotopy equivalence}
\newcommand{\grc}[1]{#1^{\mathrm{gr}}} 

\newcommand{\cA}{\cat{A}}
\newcommand{\cB}{\cat{B}}
\newcommand{\cC}{\cat{C}}
\newcommand{\cD}{\cat{D}}

\newcommand{\cS}{\cat{S}}
\newcommand{\cT}{\cat{T}}

\newcommand{\dgAlg}{\cat{dgAlg}} 
\newcommand{\dgAlgn}{\cat{dgAlg^n}} 
\newcommand{\dgCat}{\cat{dgCat}} 
\newcommand{\dgCatn}{\cat{dgCat^n}} 
\newcommand{\dgCatu}{\cat{dgCat^u}} 
\newcommand{\dgcoCatn}{\cat{dgcoCat^n}} 
\newcommand{\dgCatc}{\cat{dgCat^c}} 
\newcommand{\AAlg}{\cat{A_\infty Alg}} 
\newcommand{\AAlgn}{\cat{A_\infty Alg^n}} 
\newcommand{\AAlgu}{\cat{A_\infty Alg^u}} 
\newcommand{\AAlgc}{\cat{A_\infty Alg^c}} 
\newcommand{\ACat}{\cat{A_\infty Cat}} 
\newcommand{\ACathp}{\cat{A_\infty Cat_{hp}}}
\newcommand{\ACatn}{\cat{A_\infty Cat^n}} 
\newcommand{\ACatc}{\cat{A_\infty Cat^c}} 
\newcommand{\ACatu}{\cat{A_\infty Cat^u}} 
\newcommand{\ACatuhp}{\cat{A_\infty Cat^u_{hp}}}
\newcommand{\ACathps}{\cat{A_\infty Cat_{hps}}}
\newcommand{\Hqe}{\Ho{\dgCat}} 

\newcommand{\HoACat}{\Ho{\ACat}} 
\newcommand{\HoACatu}{\Ho{\ACatu}} 
\newcommand{\HoACatc}{\Ho{\ACatc}} 
\newcommand{\InfdgCat}{\Inf{\dgCat}}
\newcommand{\InfdgCatu}{\Inf{\dgCatu}}
\newcommand{\InfdgCatc}{\Inf{\dgCatc}}
\newcommand{\InfACat}{\Inf{\ACat}}
\newcommand{\InfACatu}{\Inf{\ACatu}}
\newcommand{\InfACatc}{\Inf{\ACatc}}

\newcommand{\InfdgAlg}{\Inf{\dgAlg}}
\newcommand{\InfdgAlgn}{\Inf{\dgAlgn}}
\newcommand{\InfAAlg}{\Inf{\AAlg}}
\newcommand{\InfAAlgu}{\Inf{\AAlgu}}
\newcommand{\InfAAlgc}{\Inf{\AAlgc}}
\newcommand{\InfAAlgn}{\Inf{\AAlgn}}

\newcommand{\QACat}{\ACat\quot{\hiso}} 
\newcommand{\QACathp}{\ACathp\quot{\hiso}}
\newcommand{\QACatu}{\ACatu\quot{\hiso}} 
\newcommand{\QACatuhp}{\ACatuhp\quot{\hiso}}
\newcommand{\QACathps}{\ACathps\quot{\hiso}}
\newcommand{\ACatdg}{\cat{A_\infty Cat_{dg}}} 
\newcommand{\HoACatdg}{\Ho{\ACatdg}} 

\newcommand{\FunA}{\cat{Fun}_{\ACat}}

\newcommand{\FunAu}{\cat{Fun}_{\ACatu}}

\newcommand{\FunAn}{\cat{Fun}_{\ACatn}}

\newcommand{\fF}{\fun{F}}
\newcommand{\fG}{\fun{G}}
\newcommand{\fH}{\fun{H}}
\newcommand{\fI}{\fun{I}}
\newcommand{\fJ}{\fun{J}}
\newcommand{\fK}{\fun{K}}

\newcommand{\fN}{\fun{N}}

\newcommand{\fS}{\fun{S}}
\newcommand{\fT}{\fun{T}}
\newcommand{\fU}{\fun{U}}
\newcommand{\fdgA}{\fI} 
\newcommand{\fdgu}{\fJ} 
\newcommand{\fuc}{\fK} 
\newcommand{\fudg}{\fT} 
\newcommand{\fudga}{\fudg'} 
\newcommand{\fcu}{\fS} 
\newcommand{\fcue}{\tilde\fcu} 
\newcommand{\fdgAn}{\fdgA^{\fun{n}}} 
\newcommand{\VdB}{\fU} 
\newcommand{\VdBn}{\VdB^{\fun{n}}} 
\newcommand{\dgmf}[1]{\tilde{#1}} 
\newcommand{\fnc}{\fH} 

\newcommand{\ncb}[1][]{\alpha_{#1}} 
\newcommand{\nbc}[1][]{\beta_{#1}} 
\newcommand{\ncbi}[1][]{\gamma_{#1}} 
\newcommand{\ntV}[1][]{\rho_{#1}} 
\newcommand{\nfV}[1][]{\sigma_{#1}} 
\newcommand{\nqV}[1][]{\pi_{#1}} 
\newcommand{\nudg}[1][]{\eta_{#1}} 
\newcommand{\nudga}[1][]{\eta'_{#1}} 
\newcommand{\nudgb}[1][]{\eta''_{#1}} 
\newcommand{\nudge}[1][]{\tilde\eta_{#1}} 
\newcommand{\nuc}[1][]{\varepsilon_{#1}} 

\newcommand{\hp}{^\mathrm{hp}} 
\newcommand{\hps}{^\mathrm{hps}} 
\newcommand{\unit}{e} 
\newcommand{\W}[1][]{\mathcal{W}_{#1}} 
\newcommand{\dg}{\mathrm{dg}}
\newcommand{\ai}{\mathrm{A}_\infty}
\newcommand{\pretr}[1][]{\mathrm{pretr}_{#1}}
\newcommand{\Yon}[1]{\fun{Y}_{#1}} 
\newcommand{\rep}[1]{\cat{R}_{#1}} 
\newcommand{\cod}{\mu}
\newcommand{\com}{\Delta}
\newcommand{\degp}{\deg'}
\newcommand{\siso}{\iota}
\newcommand{\pro}{\xi}
\newcommand{\nat}{\theta}
\newcommand{\enat}{\widetilde{\nat}}
\newcommand{\inat}{\overline{\nat}}
\newcommand{\ext}[1]{\widehat{#1}}

\begin{document}

	\title[Localizations of the categories of $A_\infty$ categories and internal Homs]{Localizations of the categories of $A_\infty$ categories and internal Homs over a ring}

	\author[A.~Canonaco]{Alberto Canonaco}
\address{A.C.: Dipartimento di Matematica ``F.\ Casorati''\\
        Universit{\`a} degli Studi di Pavia\\
        Via Ferrata 5\\
        27100 Pavia\\
        Italy}
	\email{alberto.canonaco@unipv.it}

    \author[M.~Ornaghi]{Mattia Ornaghi}
   \address{M.O.: Dipartimento di Matematica ``F.~Enriques''\\Universit\`a degli Studi di Milano\\Via Cesare Saldini 50\\ 20133 Milano\\ Italy}
   \email{mattia12.ornaghi@gmail.com}
   \urladdr{\url{https://sites.google.com/view/mattiaornaghi}}
    
    \author[P.~Stellari]{Paolo Stellari}
    \address{P.S.: Dipartimento di Matematica ``F.~Enriques''\\Universit\`a degli Studi di Milano\\Via Cesare Saldini 50\\ 20133 Milano\\ Italy}
    \email{paolo.stellari@unimi.it}
    \urladdr{\url{https://sites.unimi.it/stellari}}
	
	\thanks{A.~C.~and P.~S.~are members of GNSAGA (INdAM) and were partially supported by the research project PRIN 2022 ``Moduli spaces and special varieties''.
	In addition, M.~O.~and P.~S.~were partially supported by the research project FARE 2018 HighCaSt (grant number R18YA3ESPJ) while the research by  P.~S.~was partially funded by the project ERC-2017-CoG-771507-StabCondEn.}



\begin{abstract}
We show that, over an arbitrary commutative ring, the localizations of the categories of dg categories, of  cohomologically unital, of unital and of strictly unital $A_\infty$ categories with respect to the corresponding classes of quasi-equivalences are all equivalent. The result is proven at the $\infty$-categorical level by considering the natural $\infty$-categorical models of the categories above. As an application of the techniques we develop to compare the localizations mentioned above, we provide a new proof of the existence of internal Homs for the homotopy category of dg categories in terms of the category of (strictly) unital $A_\infty$ functors. This yields a complete proof of a claim by Kontsevich and Keller.
\end{abstract}

\maketitle

\setcounter{tocdepth}{1}
\tableofcontents

\section*{Introduction}

This paper extends and, at the same time, repairs some existing results about the homotopy categories of differential graded (dg from now on) and $A_\infty$ categories. The interest in these homotopy categories has grown during the last two decades and their study has produced several remarkable results. Nonetheless most of them depend on the assumption that such categories are linear over a field. At first sight this might look like a mild assumption but, as soon as we start thinking of applications of dg or $A_\infty$ categories to algebraic or geometric problems such as deformation theory, it becomes a priority to replace the ground field with any commutative ring.

This simple observation was the main incentive to reconsider our previous results in \cite{COS} whose proofs deeply used the assumption that the categories are linear over a field. Unfortunately, the effort to generalize our previous work drew our attention to the unpleasant presence in the literature of a couple of mistakes with deep repercussions on several papers, including \cite{COS}. We will come back to them later in the introduction. For now we want to stress that our effort to find a way out of them was not only successful but provided a wide generalization of all known results along the lines that we would like to outline now.

\smallskip

Let us consider the category $\dgCat$ consisting of (small) dg categories defined over a commutative ring $\K$.  Due to the work of Tabuada \cite{Ta}, $\dgCat$ has a model category structure which allows one to consider its homotopy category $\Hqe$, which is nothing but the localization of  $\dgCat$ with respect to all quasi-equivalences. The latter being special dg functors which induce an equivalence at the homotopy level. If we replace $\dgCat$ with the corresponding category of $A_\infty$ categories, one can still consider its localization with respect to quasi-equivalences. Note that the category of $A_\infty$ categories famously does not have a model structure with limits and colimits (see, for example, \cite[Section 1.5]{COS}).

The delicate issue about $A_\infty$ categories is that they can be endowed with various notions of unit and all of them are somehow motivated by the applications. One can indeed take the category $\ACat$ of strictly unital $A_\infty$ categories. Or, alternatively, the category $\ACatu$ of unital $A_\infty$ categories. One could go further and consider the category $\ACatc$ of cohomologically unital $A_\infty$ categories. We will discuss these subtleties in detail in \autoref{subsec:units}. For now it is enough to keep in mind that we have natural faithful (the last one also full) functors
\begin{equation}\label{eq:incl}\tag{1}
\dgCat\mono\ACat\mono\ACatu\mono\ACatc.
\end{equation}

While strictly unital $A_\infty$ categories are natural generalizations of dg categories, unital 
and cohomologically unital ones are those which appear when dealing with Fukaya categories. When the ground ring is a field the categories $\ACatu$ and $\ACatc$ coincide, but over a commutative ring they are different, in general. More flexible and interesting categorical objects are obtained by considering their localizations
\begin{equation}\label{eq:hoincl}\tag{2}
\Hqe\longrightarrow\HoACat\longrightarrow\HoACatu\longrightarrow\HoACatc
\end{equation}
with respect to the corresponding classes of quasi-equivalences.

The need for a comparison between the (homotopy) category of dg categories and the one of $A_\infty$ categories is pervasive. We will mention more applications later in the introduction. For now, it is worth recalling that the core of the Homological Mirror Symmetry Conjecture, due to Kontsevich \cite{Ko}, is indeed a comparison between dg enhancements of the bounded derived category of coherent sheaves on a Calabi--Yau threefold and the Fukaya category (hence an $A_\infty$ category) on a mirror Calabi--Yau threefold.

Taking a higher categorical viewpoint and following \cite{Pas,Tan}, we can consider the $\infty$-categories
\begin{equation}\label{eq:infincl}\tag{3}
\InfdgCat\longrightarrow\InfACat\longrightarrow\InfACatu\longrightarrow\InfACatc
\end{equation}
which are the natural $\infty$-categorical enhancements of $\Hqe$, $\HoACat$, $\HoACatu$ and $\HoACatc$, respectively, with natural $\infty$-functors induced by \eqref{eq:incl} (see the beginning of \autoref{sec:infequiv} for more details). The first main result of this paper is then the following.

\begin{thmInt}\label{infequiv}
All functors in \eqref{eq:infincl} are equivalences of $\infty$-categories.
\end{thmInt}

We should note that such a result is a generalization of Theorem 1.1 in an earlier version of \cite{OT}. Unfortunately, as the authors later realized, the proof in loc.\ cit.\ turned out to be wrong. If we stick to dg or strictly unital $A_\infty$ categories which are linear over a field, \autoref{infequiv} was proved in \cite{Pas}, by using in an essential way the results in \cite{COS}. Actually, we will explain in \autoref{subsec:Pas} that the same argument as in \cite{Pas}, together with \autoref{dgAadj}, yields a proof of the part of \autoref{infequiv} which claims that $\InfdgCat$ and $\InfACat$ are equivalent, over an arbitrary commutative ring. The equivalence between the latter category and $\InfACatu$ is established in \autoref{subsec:suu}. This result has already appeared in \cite{Tan}, but here we provide a more direct and simpler proof. The equivalence between $\InfACatu$ and $\InfACatc$ is rather easily obtained in \autoref{subsec:ucu}. It should be noted that, beyond the categories appearing in \eqref{eq:incl}, one can consider the categories $\dgCatu$ and $\dgCatc$ of, respectively, unital and cohomologically unital dg categories. As a byproduct of our proofs, it turns out that also the $\infty$-categories $\InfdgCatu$ and $\InfdgCatc$ are equivalent to those in \eqref{eq:infincl}.

It is worth pointing out that, as observed in \cite[Section 1.2]{Tan}, \autoref{infequiv} combined with the results in \cite{Hau} shows that the Gepner--Haugseng’s
model for the collection of all $\infty$-categories enriched in chain complexes in \cite{GH} is equivalent to $\InfACat$. And more applications of the above result are discussed in \cite[Section 3]{Tan}.

If we pass to the homotopy categories, then \autoref{infequiv} has a very nice and useful $1$-categorical incarnation which we state as follows.

\begin{corInt}\label{1equiv}
All functors in \eqref{eq:hoincl} are equivalences of categories.
\end{corInt}

\smallskip

We will discuss later in the introduction other interesting applications of \autoref{infequiv}. Now we want to present the second important result of the paper: a new proof for the existence of internal Homs in the homotopy category $\Hqe$. In order to make this precise, recall that given two dg categories $\cA_1$ and $\cA_2$, one can form their tensor product $\cA_1\otimes\cA_2$ in $\dgCat$. In order to get a well defined tensor product in $\Hqe$ we need to \emph{derive} it by setting $\cA_1\lotimes\cA_2:=\cA_1\otimes\cA_2\hp$. Here, given a dg category $\cA$, we denote by $\cA\hp$ (respectively $\cA\hps$) a h-projective (respectively h-projective with split units) dg category with the property that $\cA\iso\cA\hp$ (respectively $\cA\iso\cA\hps$) in $\Hqe$ (see \autoref{cofhp}). Note that, in full generality, an $A_\infty$ category $\cA$ is \emph{h-projective} if the complex of morphisms $\cA(A,A')$ is such, for all $A,A'\in\cA$, and the technical property of having split units is stated in \autoref{def:splitunit}.

The  main result in \cite{To}, later reproved in \cite{CS}, shows that the tensor product $\farg\lotimes\cA_2$ has a right adjoint $\IHom(\cA_2,\farg)$ in $\Hqe$. Namely, for $\cA_1$, $\cA_2$ and $\cA_3$ in $\dgCat$, we have a dg category $\IHom(\cA_2,\cA_3)$ and a natural bijection
\begin{equation}\label{eq:intHom1}\tag{4}
\xymatrix{
	\Hqe(\cA_1\lotimes\cA_2,\cA_3) \ar@{<->}[rr]^-{1:1} & & \Hqe(\cA_1,\IHom(\cA_2,\cA_3)).
}
\end{equation}
The dg category $\IHom(\cA_2,\cA_3)$ is the \emph{internal Hom} between $\cA_2$ and $\cA_3$, and it is uniquely determined (up to isomorphism in $\Hqe$) by the bijection in \eqref{eq:intHom1}.

A direct consequence is that we get a natural bijection
\begin{equation}\label{eq:intHom2}\tag{5}
\xymatrix{
	\Hqe(\cA_,\cB) \ar@{<->}[rr]^-{1:1}& &\Iso(H^0(\IHom(\cA,\cB))),
}
\end{equation}
for all dg categories $\cA,\cB\in\dgCat$.

The astonishing fact is that, well before the appearance of \cite{To}, Kontsevich envisioned that such internal Homs should exist and could be described in terms of $A_\infty$ functors between the corresponding dg categories. Such a claim, originally mentioned in \cite{Dr}, was later recast by Keller in his ICM talk \cite{K2} (see Section 4.3 therein). Its formulation requires some additional notation which we briefly sketch here.

If $\cA,\cB\in\ACatu$, we can consider the unital $A_\infty$ category $\FunAu(\cA,\cB)$ which will be carefully defined in \autoref{subsec:equivfun} and whose objects are the unital $A_\infty$ functors from $\cA$ to $\cB$. Analogously, if $\cA,\cB\in\ACat$, we can consider the strictly unital $A_\infty$ category $\FunA(\cA,\cB)$ of stricly unital $A_\infty$ functors from $\cA$ to $\cB$. It is important to keep in mind that $\FunAu(\cA,\cB)$ and $\FunA(\cA,\cB)$ are dg categories if $\cB$ is such. The precise formulation of the claim mentioned above is then the following.

\medskip

\noindent{\bf Claim} (Kontsevich, Keller){\bf .} \emph{Let $\cA$ and $\cB$ be dg categories such that $\cA(A,A')$ is a cofibrant complex and the unit map $\K\to\cA(A,A)$ admits a retraction as a morphism of complexes, for all $A,A'\in\cA$. Then $\FunA(\cA,\cB)$ is the dg category of internal Homs between $\cA$ and $\cB$.}

\medskip

Actually we can prove the following result which implies, as a special case, the claim above (see \autoref{rmk:KK}). At the same time, due to its gorgeous generality, it provides a completely new proof of the result in \cite{To} about the existence of internal Homs.

\begin{thmInt}\label{IntHom}
Given two dg categories $\cA,\cB$, the internal Hom in $\Hqe$ between $\cA$ and $\cB$ is given by the dg category
\[
\FunAu(\cA\hp,\cB)\iso\FunA(\cA\hps,\cB),
\]
where the isomorphism is in $\Hqe$. Furthermore, we have equivalences of categories
\[
\xymatrix{
\Hqe\iso\QACatuhp\,\iso\QACathps.
}
\]
\end{thmInt}

Here we denote by $\ACatuhp$ the full subcategory of $\ACatu$ whose objects are h-projective. We can then take the quotient $\QACatuhp$ of $\ACatuhp$ with respect to the equivalence relation which states that two unital $A_\infty$ functors are (weakly) equivalent $\fF\hiso\fG$ if they are isomorphic in the $0$-th cohomology of $\FunAu(\cA,\cB)$. Similarly, $\ACathps$ denotes the full subcategory of $\ACat$ whose objects are h-projective and have split units and we can take an analogous quotient $\QACathps$ of $\ACathps$. Clearly, if $\K$ is a field, then $\ACatu=\ACatuhp$ (but $\ACathps$ is strictly contained in $\ACat$).

While the proof of the first part of \autoref{IntHom} is the content of \autoref{sec:IntHom}, the fact that $\Hqe$ is equivalent to $\QACatuhp$ and $\QACathps$ is established in \autoref{unitalhp}. It is worth pointing out that the equivalence $\Hqe\iso\QACatuhp$ (and similarly the equivalence $\Hqe\iso\QACathps$) can be seen as a categorification of \eqref{eq:intHom2}. Indeed, given two dg categories $\cA$ and $\cB$, we have a natural bijection
\[
\xymatrix{
\Hqe(\cA,\cB) \ar@{<->}[r]^-{1:1}&\QACatuhp(\cA\hp,\cB\hp)=\Iso(H^0(\IHom(\cA,\cB\hp))),
}
\]
and the latter is in natural bijection with $\Iso(H^0(\IHom(\cA,\cB)))$ by \eqref{eq:intHom1}.

\subsection*{Applications and further developments}

Our \autoref{infequiv} should be mainly seen as a foundational result that puts on solid ground the homotopy theory of $A_\infty$ categories. One issue of the category of $A_\infty$ categories which we have mentioned above is the lack of model structure which makes hard to define and study homotopy limits and colimits of $A_\infty$ categories. Our result builds a bridge that allows one to import categorical constructions from the much better behaved world of dg categories where a model structure exists. An interesting discussion on how comparing results similar to \autoref{infequiv} may be used for the purposes of computing homotopy limits and colimits can be found in \cite[Section A.4]{GPS}.

One important incarnation of the ideas mentioned above is the application of \autoref{1equiv} to the problem of existence and uniqueness of enhancements for triangulated categories. Indeed, from \autoref{1equiv} we immediately get the following useful result.

\begin{corInt}\label{cor:uniqen}
A triangulated category has a (unique) dg enhancement if and only if it has a (unique) $A_\infty$ enhancement.
\end{corInt}

Note that for most triangulated categories naturally appearing in algebra or in algebraic geometry the existence of a (dg) enhancement is very easy to prove. On the other hand, the problem of uniqueness of enhancement is usually much harder. The interest for this problem was initiated by an influential conjecture by Bondal, Larsen and Lunts in \cite{BLL} for geometric triangulated categories. The conjecture was proved by Lunts and Orlov in the seminal paper \cite{LO} and the result was then further extended in \cite{CS6} and \cite{An} (see also \cite{Ge}), up to the last and most general result in \cite{CNS}, where it was proved that the following triangulated categories have unique dg (hence also $A_\infty$) enhancements:
\begin{itemize}
\item[{\rm (a)}] $\mathrm{D}^?(\cA)$, where $\cA$ is an abelian category and $?=b,+,-,\emptyset$;
\item[{\rm (b)}] The category $\mathrm{D}_\text{qc}^?(X)$ of complexes of $\mathcal{O}_X$-modules with quasi-coherent cohomology and the category $\mathrm{Perf}(X)$ of perfect complexes on a quasi-compact and quasi-separated scheme, for $?=b,+,-,\emptyset$.
\end{itemize}

Another interesting feature of \autoref{infequiv} is that its proof can be very easily adapted to prove a similar statement for algebras. More precisely, we can consider the categories $\AAlg$, $\AAlgu$, $\AAlgc$ and $\AAlgn$ of strictly unital, unital, cohomologically unital and non-unital $A_\infty$ algebras, respectively. Similarly, $\dgAlg$ and $\dgAlgn$ will be the categories of strictly unital and non-unital dg algebras, respectively. If $\cC$ is any of these categories, we denote by $\Ho{\cC}$ its localization with respect to quasi-isomorphisms, and by $\Inf{\cC}$ the $\infty$-categorical enhancement of $\Ho{\cC}$. Here it is important to observe that, unlike quasi-equivalences for (dg or $A_\infty$) categories, quasi-isomorphisms for (dg or $A_\infty$) algebras make perfect sense also in the non-unital case. We then immediately get the following result (see \autoref{rmk:algebras} for the corresponding discussion).

\begin{thmInt}\label{infalg}
The natural functors
\begin{gather*}
\dgAlg\mono\AAlg\mono\AAlgu\mono\AAlgc, \\
\dgAlgn\mono\AAlgn
\end{gather*}
induce equivalences of $\infty$-categories
\begin{gather*}
\InfdgAlg\longrightarrow\InfAAlg\longrightarrow\InfAAlgu\longrightarrow\InfAAlgc, \\
\InfdgAlgn\longrightarrow\InfAAlgn.
\end{gather*}
\end{thmInt}

We avoid to state the analogous of \autoref{1equiv} for algebras, which of course can be directly deduced from \autoref{infalg}. We point out that, in the non-unital case and when $\K$ is a field of characteristic $0$, such a result also follows from \cite[Theorem 11.4.8]{LV}, which is formulated in the more general language of operads.

\medskip

Let us now discuss the applications of \autoref{IntHom} and let us go back to the original description in \cite{To} of the internal Homs in terms of special dg bimodules. The fact that dg bimodules can be easily composed via convolutions, led To\"en to state the following as an open question in the introduction to \cite{To}.

\medskip

\noindent{\bf Question} (To\"en){\bf .} \emph{Is the bijection \eqref{eq:intHom2} compatible with compositions?}

\medskip

Our new description of the internal Homs and the full power of \autoref{IntHom} immediately yields a positive answer to this question (see \autoref{rmk:compositions} for a detailed discussion). Finally, we hope that our new description might clarify the still partially mysterious relation between objects in $\Iso(H^0(\IHom(\cA,\cB)))$ and Fourier--Mukai kernels of exact functors when $\cA$ and $\cB$ are dg enhancements of the derived categories of quasi-coherent sheaves on two quasi-compact and separated schemes. In \cite[Section 8.3]{To} it was claimed that such a relation should be encoded in a natural way by the bijection \eqref{eq:intHom2}. However, a (far from trivial) proof of this claim is available only under restrictive assumptions on the schemes (see \cite{LS}). 

\subsection*{Related work}

The first comparison has to be made with our previous paper \cite{COS}. Besides the obvious observation that our new results imply essentially all the ones in \cite{COS}, we should go back to our first claim in the introduction: the fact that this paper corrects and overcomes some mistakes in the literature.

It was proved by Lef\`evre-Hasegawa \cite[Proposition 3.2.4.1]{LH} for $A_\infty$ algebras and later by Seidel \cite[Lemma 2.1]{Sei} for $A_\infty$ categories that any cohomologically unital $A_\infty$ algebra or category can be replaced with a strictly unital one, at least when we work over a field. Similarly, \cite[Theorem 3.2.2.1]{LH} and \cite[Remark 2.2]{Sei} claim that the same is true for functors: an $A_\infty$ functor between strictly unital $A_\infty$ categories can be replaced by a strictly unital one, up to homotopy. Unfortunately, after carefully thinking about these claims, one realizes that none of them can be true in this generality for trivial reasons. And the falsity of the first claim about categories (and algebras) was indeed later observed  by Seidel in an erratum to \cite{Sei}.

While many technical parts of \cite{COS} remain valid (and will also be used in this paper) others, heavily relying on the two claims above, have to be revisited. For example, \cite[Proposition 2.5]{COS} is easily seen to be false as soon as we consider $A_\infty$ categories $\cA$ such that the complex $\cA(A,A)$ has trivial cohomology but it is not trivial, for some some $A$ in $\cA$. This produces a cascade of problems in the proof of \cite[Theorem A]{COS}, some of which can be overcome by readjusting the arguments, while some of them needs the new (and at the same time more general) approach which we adopt in the present paper. A careful comparison shows that the new \autoref{1equiv} replaces the old one for most of its parts once we recall that the categories of cohomologically unital and of unital $A_\infty$ categories are the same, over a field. There is only one claim in \cite[Theorem A]{COS} that is not covered by our new results: the equivalence between $\Hqe$ and $\QACat$. In fact we expect that such an equivalence does not hold, as \autoref{1equiv} shows that $\Hqe$ is equivalent to $\QACathps$ and it is easy to see that the inclusion of $\QACathps$ in $\QACathp\,=\QACat$ is not an equivalence (see \autoref{suwrong}).

Moreover, for similar reasons, the description of the internal Homs in terms of strictly unital $A_\infty$ functors has to be corrected or replaced by the one which uses unital $A_\infty$ functors. The result is that the new \autoref{IntHom} replaces and generalizes the old \cite[Theorem B]{COS}.

Finally, as we have already explained before, \autoref{infequiv} for categories linear over a commutative ring provides a complete proof of Theorem 1.1 in an earlier version of \cite{OT}. The latter result together with \autoref{IntHom} gives then access to the many very interesting applications discussed in the second part of \cite{OT} (see, in particular, Sections 4 and 5 therein) and in \cite[Section 1.3]{Tan}.

\subsection*{Plan of the paper}

In \autoref{sec:introAinfty} we briefly recall the basic definitions and constructions which are used all along the paper. We refer to the existing literature for more details but an issue that we try to analyze carefully is the difference between the various notions of unit for $A_\infty$ categories (see \autoref{subsec:units}).
 
In \autoref{sec:adjUI} we show the existence of a crucial pair of adjoint functors between the category of non-unital dg categories $\dgCatn$ and $\ACatn$ (for later use in \autoref{sec:IntHom}, the key step of the proof needs to be treated in a more general setting, which makes \autoref{sec:adjUI} the more technical part of the paper). The analysis has to be refined in \autoref{sec:infequiv} in order to deal with strictly unital, unital and cohomologically unital $A_\infty$ categories, thus proving \autoref{infequiv}. 

 As for \autoref{IntHom}, the proof of the last part is carried out in \autoref{sec:1equiv}. The first part of \autoref{IntHom} is proved in \autoref{sec:IntHom}, after a preliminary discussion about multifunctors in \autoref{subsec:multifun}. 
 
 \subsection*{Notation and conventions}
 
 We assume that a universe containing an infinite set is fixed, and we will simply call sets the members of this universe. In general the collection of objects of a category need not be a set: we will always specify if we are requiring this extra condition.
 
 We work over a commutative ring $\K$.  We will always assume that the collection of objects in a $\K$-linear category is a set.

 The shift by an integer $n$ of a graded $\K$-module $M=\bigoplus_{i\in\ZZ}M^i$ will be denoted by $\sh[n]{M}=\bigoplus_{i\in\ZZ}M^{n+i}$. If $x\in M$, we will often write $\sh[n]{x}$ to denote the same element in $\sh[n]{M}$. If $x$ is homogeneous, say $x\in M^i$, then we set $\deg(x):=i$ and $\degp(x):=i+1$.

We recall the Koszul sign rule: if $f\colon M\to M'$ and $g\colon N\to N'$ are morphisms of graded $\K$-modules, with $g$ homogeneous, then $f\otimes g\colon M\otimes N\to M'\otimes N'$ maps $x\otimes y$, with $x$ homogeneous, to $(-1)^{\deg(g)\deg(x)}f(x)\otimes g(y)$.

Complexes (or dg modules) are cohomological (namely, the differential has degree $+1$).

\section{Preliminaries on $A_\infty$ categories and functors}\label{sec:introAinfty}

In this section we provide a concise introduction to $A_\infty$ categories and functors. In the whole paper the subtle relation between the various notions of unit is crucial. We provide here the basic definitions and properties which will be used later.

As in \cite{COS}, we will follow the sign conventions in \cite{LH}, which are different from (but equivalent to) those used in other references, like \cite{BLM} and \cite{Sei}. In particular, given graded $\K$-modules $M_1,\dots,M_i,N$, a $\K$-linear map $f\colon M_i\otimes\cdots\otimes M_1\to N$ of degree $n$ is identified with the $\K$-linear map
\begin{equation}\label{eq:signshift}
(-1)^{n+i-1}\siso_{N}^{-1}\comp f\comp(\siso_{M_i}\otimes\cdots\otimes\siso_{M_1})\colon\sh{M_i}\otimes\cdots\otimes\sh{M_1}\to\sh{N}
\end{equation}
of degree $n+i-1$, where $\siso_M\colon\sh{M}\to M$ is the natural isomorphism of degree $1$, for every graded $\K$-module $M$.

\subsection{Non-unital $A_\infty$ categories, functors and natural transformations}\label{subsec:Ainftycat}

We start by recalling the explicit definitions of non-unital $A_\infty$ categories, functors and (pre)natural transformations. A conceptual explanation of the otherwise mysterious formulas appearing in this section will be given in \autoref{subsec:barcobar}.

\begin{definition}\label{nucat}
A \emph{non-unital $A_\infty$ category} $\cA$ consists of a set of objects $\Ob{\cA}$, of graded $\K$-modules $\cA(A,A')$ for every $A,A'\in\cA$ and of $\K$-linear maps of degree $2-i$
\begin{equation}\label{eq:catmaps}
\m{i}=\m[\cA]{i}\colon\cA(A_{i-1},A_i)\otimes\cdots\otimes\cA(A_{0},A_1)\to\cA(A_0,A_i),
\end{equation}
for every $i>0$ and every $A_0,\ldots,A_i\in\cA$. The maps must satisfy the \emph{$A_\infty$ associativity relations}
\begin{equation}\label{eq:catrel}
\sum_{k=1}^n\sum_{i=0}^{n-k}(-1)^{i+k(n-i-k)}\m{n-k+1}\comp(\id^{\otimes n-i-k}\otimes\m{k}\otimes\id^{\otimes i})=0
\end{equation}
for every $n>0$.
\end{definition}

In particular, \eqref{eq:catrel} with $n=1$ shows that $\m{1}$ defines a differential on each $\cA(A,A')$, which will always be regarded as a complex in this way. It is also important to observe that $\m{1}$ satisfies the graded Leibniz rule with respect to the composition defined by $\m{2}$ (by the case $n=2$) and that $\m{2}$ is associative, up to a homotopy defined by $\m{3}$ (by the case $n=3$). This implies that we obtain the non-unital graded \emph{cohomology category} $H(\cA)$ of $\cA$ such that $\Ob{H(\cA)}=\Ob{\cA}$,
\[
H(\cA)(A,A')=\bigoplus_i H^i\bigl(\cA(A,A')\bigr)
\]
for every $A,A'\in\cA$ and (associative) composition induced from $\m{2}$.

\begin{definition}\label{nufun}
A \emph{non-unital $A_\infty$ functor} $\fF\colon\cA\to\cB$ between two non-unital $A_\infty$ categories $\cA$ and $\cB$ is a collection $\fF=\{\fF^i\}_{i\geq 0}$, where $\fF^0\colon\Ob{\cA}\to\Ob{\cB}$ is a map of sets and
\begin{equation}\label{eq:funmaps}
\fF^i\colon\cA(A_{i-1},A_i)\otimes\cdots\otimes\cA(A_0,A_1)\to\cB(\fF^0(A_0),\fF^0(A_i)),
\end{equation}
for $i>0$, are $\K$-linear maps of degree $1-i$, for every $A_0,\ldots,A_i\in\cA$. The maps must satisfy the following relations
\begin{multline}\label{eq:funrel}
\sum_{k=1}^n\sum_{i=0}^{n-k}(-1)^{i+k(n-i-k)}\fF^{n-k+1}\comp(\id^{\otimes n-i-k}\otimes\m[\cA]{k}\otimes\id^{\otimes i})\\
=\sum_{\substack{i_1+\cdots+i_r=n\\i_1,\dots,i_r>0}}(-1)^{\sum_{t=1}^{r-1}\sum_{u=t+1}^r(1-i_t)i_u}\m[\cB]{r}\comp(\fF^{i_r}\otimes\cdots\otimes\fF^{i_1}),
\end{multline}
for every $n>0$. A non-unital $A_\infty$ functor $\fF$ is \emph{strict} if $\fF^i=0$ for every $i>1$.
\end{definition}

From \eqref{eq:funrel} with $n=1$ we see that $\fF^1$ commutes with the differentials $\m{1}$. Moreover, $\fF^1$ preserves the compositions $\m{2}$, up to a homotopy defined by $\fF^2$ (by the case $n=2$). It follows that $\fF^0$ and $\fF^1$ induce a non-unital graded functor $H(\fF)\colon H(\cA)\to H(\cB)$.

\begin{remark}
A non-unital $A_\infty$ category $\cA$ such that $\m{i}=0$ for all $i>2$ is called a \emph{non-unital dg category}; for such categories $\m{1}$ and $\m{2}$ are usually denoted by $d$ and $\comp$. A strict non-unital $A_\infty$ functor $\fF$ between two dg categories is called a \emph{non-unital dg functor}; in this case one often writes $\fF$ instead of $\fF^0$ or $\fF^1$. There is a category $\ACatn$ (with objects the non-unital $A_\infty$ categories and morphisms the non-unital $A_\infty$ functors) which contains as a subcategory $\dgCatn$ (with objects the non-unital dg categories and morphisms the non-unital dg functors). While the composition in $\dgCatn$ is the obvious one, the composition in $\ACatn$ is more subtle (see \autoref{subsec:barcobar}); however, we will not need its explicit definition.
\end{remark}

\begin{definition}
Given $\fF,\fG\colon\cA\to\cB$ in $\ACatn$, a \emph{prenatural transformation} $\nat\colon\fF\to\fG$ of degree $p$ is given by $\K$-linear maps of degree $p-i$
\begin{equation}\label{eq:natmaps}
\nat^i\colon\cA(A_{i-1},A_i)\otimes\cdots\otimes\cA(A_0,A_1)\to\cB\bigl(\fF^0(A_0),\fG^0(A_i)\bigr)
\end{equation}
for every $i\ge0$ and every $A_0,\ldots,A_i\in\cA$. We say that $\nat$ is a \emph{natural transformation} if
\begin{multline}\label{eq:natrel}
\sum_{k=1}^n\sum_{i=0}^{n-k}(-1)^{i+k(n-i-k)}\nat^{n-k+1}\comp(\id^{\otimes n-i-k}\otimes\m[\cA]{k}\otimes\id^{\otimes i}) \\
+\sum_{\substack{i_1+\dots+i_r+k+j_1+\cdots+j_s=n\\i_1,\dots,i_r,j_1,\dots,j_s>0,k\ge0}}(-1)^{p+r(p-1)+\sum_{t=1}^r(1-i_t)(n-\sum_{u=1}^{t-1}i_u)+(p-k)\sum_{t=1}^sj_t+\sum_{t=1}^{s-1}\sum_{u=t+1}^s(1-j_t)j_u} \\
\m[\cB]{r+s+1}\comp(\fG^{j_s}\otimes\cdots\otimes\fG^{j_1}\otimes\nat^k\otimes\fF^{i_r}\otimes\cdots\otimes\fF^{i_1})=0
\end{multline}
for every $n\ge0$.
\end{definition}

Observe that $\nat^0$ can be identified with a collection of elements $\nat^0_A\in\cB\bigl(\fF^0(A),\fG^0(A)\bigr)^p$ for every $A\in\cA$. Moreover, \eqref{eq:natrel} with $n=0$ shows that these elements are closed, whereas the case $n=1$ implies that their images in cohomology define a natural transformation $H(\nat)\colon H(\fF)\to H(\fG)$ of degree $p$.

\subsection{Reminder on the bar and cobar constructions}\label{subsec:barcobar}

This section is a quick reminder about the bar and cobar constructions. In \cite[Sections 1.2 and 1.3]{COS} the reader can find definitions and properties of some notions which are not recalled here, like those of (graded or dg) quiver, cocategory and cofunctor. For a more detailed presentation see also \cite{BLM}.

We denote by $\dgcoCatn$ the category whose objects are non-unital cocomplete dg cocategories and whose morphisms are non-unital dg cofunctors.

Given $\cA\in\ACatn$, the \emph{bar construction} $\Bi(\cA)\in\dgcoCatn$ associated to $\cA$ is simply defined to be $\Tcr{\sh{\cA}}$ (where $\cA$ is viewed as a graded quiver) as a non-unital graded cocategory. As for the differential, an arbitrary choice of maps $\m[\cA]{i}$ as in \eqref{eq:catmaps} determines a morphism of graded quivers $\Tcr{\sh{\cA}}\to\sh{\cA}$ of degree $1$ (recall \eqref{eq:signshift}), which extends uniquely to a $(\id_{\Tcr{\sh{\cA}}},\id_{\Tcr{\sh{\cA}}})$-coderivation $d_\cA\colon\Tcr{\sh{\cA}}\to\Tcr{\sh{\cA}}$ of degree $1$. It is easy to see that $d_\cA\comp d_\cA=0$ if and only if \eqref{eq:catrel} holds for every $n>0$, in which case we set $\Bi(\cA):=(\Tcr{\sh{\cA}},d_\cA)$.

Similarly, $\fF\colon\cA\to\cB$ in $\ACatn$ induces $\Bi(\fF)\colon\Bi(\cA)\to\Bi(\cB)$ in $\dgcoCatn$. More precisely, an arbitrary choice of maps $\fF^0\colon\Ob{\cA}\to\Ob{\cB}$ and $\fF^i$ for $i>0$ as in \eqref{eq:funmaps} determines a morphism of graded quivers $\Tcr{\sh{\cA}}\to\sh{\cB}$ of degree $0$, which extends uniquely to a graded cofunctor $\ext{\fF}\colon\Tcr{\sh{\cA}}\to\Tcr{\sh{\cB}}$. Then one can check that $d_\cB\comp\ext{\fF}=\ext{\fF}\comp d_\cA$ if and only if \eqref{eq:funrel} holds for every $n>0$, in which case we set $\Bi(\fF):=\ext{\fF}$. Moreover, the composition in $\ACatn$ is defined in such a way that
\[
\Bi\colon\ACatn\to\dgcoCatn
\]
is a functor, which actually turns out to be fully faithful.

Finally, given $\fF,\fG\colon\cA\to\cB$ in $\ACatn$ (more generally, $\fF$ and $\fG$ could be given by arbitrary maps $\fF^0,\fG^0\colon\Ob{\cA}\to\Ob{\cB}$ and $\fF^i,\fG^i$ for $i>0$ as in \eqref{eq:funmaps}), a prenatural transformation $\nat\colon\fF\to\fG$ of degree $p$ determines $\K$-linear maps $\Tc{\sh{\cA}}(A,A')\to\sh{\cB}\bigl(\fF^0(A),\fG^0(A')\bigr)$ of degree $p-1$ for every $A,A'\in\cA$, which extend uniquely to a $(\ext{\fF},\ext{\fG})$-coderivation $\ext{\nat}\colon\Tc{\sh{\cA}}\to\Tcr{\sh{\cB}}$ of degree $p-1$. Here we still denote by $\ext{\fF},\ext{\fG}\colon\Tc{\sh{\cA}}\to\Tcr{\sh{\cB}}$ the extensions by $0$ of $\ext{\fF}$ and $\ext{\fG}$. Again, it can be easily proved that $d_\cB\comp\ext{\nat}+(-1)^p\ext{\nat}\comp d_\cA=0$ (where we still denote by $d_\cA\colon\Tc{\sh{\cA}}\to\Tc{\sh{\cA}}$ the extensions by $0$ of $d_\cA$) if and only if $\nat$ is a natural transformation.

\begin{remark}\label{natfun}
Given $\fF\colon\cA\to\cB$ in $\ACatn$ and a prenatural transformation $\nat\colon\fF\to\fF$ of degree $1$ such that $\nat^0=0$, we can define $\fG^0:=\fF^0$ and $\fG^i:=\fF^i+\nat^i$ for $i>0$. Then we can regard $\nat$ as a prenatural transformation $\fF\to\fG$, and it is not difficult to show that in this way $\ext{\fG}=\ext{\fF}+\ext{\nat}$. This clearly implies that $\fG$ is a non-unital $A_\infty$ functor if and only if $\nat\colon\fF\to\fG$ is a natural transformation.
\end{remark}

As a matter of notation, we set
\[
\B:=\Bi\rest{\dgCatn}\colon\dgCatn\to\dgcoCatn,
\]
which is a faithful (but not full) functor. Dually, the \emph{cobar construction} yields a faithful (but not full) functor
\[
\coB\colon\dgcoCatn\to\dgCatn.
\]
In particular, for $\cC\in\dgcoCatn$, $\coB(\cC)$ is simply defined to be $\Tar{\sh[-1]{\cC}}$ as a non-unital graded category, with differential induced from the differential and the cocomposition in $\cC$.

By \cite[Proposition 1.21]{COS} there is an adjunction
\[
\adjpair{\coB}{\dgcoCatn}{\dgCatn}{\B},
\]
with counit denoted by $\ncb\colon\coB\comp\B\to\id_{\dgCatn}$ and unit denoted by $\nbc\colon\id_{\dgcoCatn}\to\B\comp\coB$. Since $\Bi$ is fully faithful, for every $\cA\in\ACatn$ there exists unique $\ncbi[\cA]\in\ACatn(\cA,\coB(\Bi(\cA)))$ such that
\[
\nbc[\Bi(\cA)]=\Bi(\ncbi[\cA])\colon\Bi(\cA)\to\Bi(\coB(\Bi(\cA)))=\B(\coB(\Bi(\cA))).
\]
Denoting by $\fdgAn\colon\dgCatn\to\ACatn$ the inclusion functor and setting
\[
\VdBn:=\coB\comp\Bi\colon\ACatn\to\dgCatn,
\]
it is clear that the $A_\infty$ functors $\ncbi[\cA]\colon\cA\to\coB(\Bi(\cA))=\VdBn(\cA)$ (for $\cA\in\ACatn$) define a natural transformation $\ncbi\colon\id_{\ACatn}\to\fdgAn\comp\VdBn$.

\subsection{Notions of unit}\label{subsec:units}

Now we need to discuss the various notions of unit which will be used in the rest of the paper.

\begin{definition}\label{cucat}
A \emph{cohomologically unital $A_\infty$ category} is a non-unital $A_\infty$ category $\cA$ such that $H(\cA)$ is a category (i.e.\ $H(\cA)$ is unital).
\end{definition}

\begin{remark}\label{cunit}
If $\cA$ is a cohomologically unital $A_\infty$ category, then for every $A\in\cA$ there exists (unique up to coboundaries) a closed degree $0$ element $\unit_A\in\cA(A,A)$ representing $\id_A\in H(\cA)(A,A)$. We will say that $\unit_A$ is a \emph{cohomological unit} of $A$.
\end{remark}

\begin{definition}[{\cite[Definition 7.3 and Lemma 7.4]{Lyu}}]\label{ucat}
An {$A_\infty$} category $\cA$ is \emph{unital} if it is cohomologically unital and the following morphisms of complexes
\[
\m{2}(\farg\otimes\unit_A)\colon\cA(A,A')\to\cA(A,A') \qquad \m{2}(\unit_A\otimes\farg)\colon\cA(A',A)\to\cA(A',A)
\]
are homotopic to the identity for every $A,A'\in\cA$, and for some (hence for every) cohomological unit $\unit_A$ of $A$. In this case the cohomological units of $\cA$ can be simply called \emph{units}.
\end{definition}

\begin{definition}\label{sucat}
A \emph{strictly unital $A_\infty$ category} is a non-unital $A_\infty$ category $\cA$ such that for every $A\in\cA$ there exists (unique) a degree $0$ morphisms $\id_A\in\cA(A,A)$ (called a \emph{strict unit} of $A$) satisfying the following properties:
\begin{enumerate}
\item $\m{2}(\farg\otimes\id_A)=\id_{\cA(A,A')}$ and $\m{2}(\id_A\otimes\farg)=\id_{\cA(A',A)}$ for every $A,A'\in\cA$;
\item $\m{i}(f_i\otimes\cdots\otimes f_1)=0$ if $i\ne 2$ and $f_j=\id_A$ for some $j\in\{1,\dots,i\}$ and some $A\in\cA$.
\end{enumerate}
\end{definition}

For a non-unital $A_\infty$ category $\cA$, its \emph{augmentation} is the strictly unital $A_\infty$ category $\aug{\cA}$ such that $\Ob{\aug{\cA}}=\Ob{\cA}$ and
\[
\aug{\cA}(A,A')=
\begin{cases}
\cA(A,A') & \text{if $A\ne A'$} \\
\cA(A,A')\oplus\K\,1_A & \text{if $A=A'$,}
\end{cases}
\]
with $\m[\aug{\cA}]{i}$ the unique extension of $\m[\cA]{i}$ such that the additional morphisms $1_A$ is the unit of $A$ in $\aug{\cA}$, for every $A\in\cA$ and every $i>0$. To avoid confusion, when $\cA$ is strictly unital, the strict unit in $\cA$ is denoted by $\id_A$ while the one in $\aug{\cA}$ is $1_A$, for every $A\in\cA$. 

Similarly, we get \emph{cohomologically unital dg categories}, \emph{unital dg categories} and \emph{strictly unital dg categories}. In accordance to the existing literature, strictly unital dg categories (respectively functors) will be simply referred to as dg categories (respectively functors).

\begin{ex}\label{tensor}
(i) In the special case of (dg or $A_\infty$) categories with only one object, then, for obvious reasons, we will talk about (dg or $A_\infty$) algebras.

(ii) Given two (non-unital, cohomologically unital, unital or strictly unital) dg categories $\cA$ and $\cB$, we can define a (non-unital, cohomologically unital, unital or strictly unital) dg category $\cA\otimes\cB$, which is the tensor product of $\cA$ and $\cB$. Its objects are the pairs $(A,B)$ with $A\in\cA$ and $B\in\cB$, while $(\cA\otimes\cB)\bigl((A,B),(A',B')\bigr)=\cA(A,A')\otimes\cB(B,B')$. If $\cA$ and $\cB$ are $A_\infty$ categories, then defining an appropriate tensor product is a more delicate issue which will be discussed in \cite{Orn1}.
\end{ex}

There is also a notion of homotopy unital $A_\infty$ category (which will not be used in this paper), such that the following implications hold for $A_\infty$ categories (see \cite[Section 8.12.]{Lyu}):
\[
\text{strictly unital}\implies\text{homotopy unital}\implies\text{unital}\implies\text{cohomologically unital}
\]

\begin{remark}\label{u=cu}
If $\K$ is a field, then an $A_\infty$ category is unital if and only if it is cohomologically unital. This is simply due to the fact that, over a field, two morphisms of complexes are homotopic if they induce the same map in cohomology.
\end{remark}

Of course, there are also the corresponding notions of strictly unital, unital and cohomologically unital $A_\infty$ functors (see \cite[pp 23]{Sei} and \cite[Definition 8.1. and Proposition 8.2.]{Lyu}).

\begin{definition}\label{ufun}
Let $\fF\colon\cA\to\cB$ be a non-unital $A_\infty$ functor.

(i) $\fF$ is \emph{(cohomologically) unital} if $\cA$ and $\cB$ are (cohomologically) unital and $H(\fF)$ is unital.

(ii) $\fF$ is \emph{strictly unital} if $\cA$ and $\cB$ are strictly unital and the following properties are satisfied:
\begin{enumerate}
\item $\fF^1(\id_A)=\id_{\fF^0(A)}$ for every $A\in\cA$;
\item $\fF^i(f_i\otimes\cdots\otimes f_1)=0$ if $i>1$ and $f_j=\id_A$ for some $j\in\{1,\dots,i\}$ and some $A\in\cA$.
\end{enumerate}
\end{definition}


The following definition is only partially standard.

\begin{definition}
A non-unital $A_\infty$ functor $\fF\colon\cA\to\cB$ is a \emph{quasi-isomorphism} (respectively a \emph{\htpiso}) if $\fF^0$ is bijective and $\fF^1\colon\cA(A,A')\to\cB(\fF^0(A),\fF^0(A'))$ is a quasi-isomorphism (respectively a homotopy equivalence) of complexes for every $A,A'\in\cA$.

When $\cB$ is cohomologically unital (respectively unital), $\fF\colon\cA\to\cB$ is a \emph{quasi-equivalence} (respectively a \emph{\htpequiv}) if $H(\fF)$ is essentially surjective and $\fF^1\colon\cA(A,A')\to\cB(\fF^0(A),\fF^0(A'))$ is a quasi-isomorphism (respectively a homotopy equivalence) of complexes for every $A,A'\in\cA$.
\end{definition}

\begin{remark}\label{unithtpiso}
Clearly every {\htpiso} (respectively \htpequiv) is a quasi-isomorphism (respectively quasi-equivalence), and the converse holds if $\K$ is a field. It is also easy to see that, if a non-unital $A_\infty$ functor $\fF\colon\cA\to\cB$ is a quasi-isomorphism (respectively a \htpiso), then $\cA$ is cohomologically unital (respectively unital) if and only if $\cB$ is cohomologically unital (respectively unital), and in this case $\fF$ is cohomologically unital (respectively unital).
\end{remark}

We will denote by $\ACat$ (respectively $\ACatu$, respectively $\ACatc$) the subcategory of $\ACatn$ whose objects are strictly unital (respectively unital, respectively cohomologically unital) $A_\infty$ categories and whose morphisms are strictly unital (respectively cohomologically unital) $A_\infty$ functors. Similarly, $\dgCat$ (respectively $\dgCatu$, respectively $\dgCatc$) denotes the subcategory of $\dgCatn$ whose objects are strictly unital (respectively unital, respectively cohomologically unital) dg categories and whose morphisms are strictly unital (respectively cohomologically unital) dg functors. Moreover, $\ACatdg$ will be the full subcategory of $\ACat$ whose objects are dg categories.

\begin{remark}
In any case $\ACatu$ (respectively $\dgCatu$) is a full subcategory of $\ACatc$ (respectively $\dgCatc$). Moreover, if $\K$ is a field, it follows from \autoref{u=cu} that $\ACatu=\ACatc$ and $\dgCatu=\dgCatc$.
\end{remark}

In order to study the relation between $\ACatu$ and $\dgCat$, later we will need the following result (which comes from an $A_\infty$ categorical version of Yoneda's lemma).

\begin{lem}\label{Yoneda}
Given $\cA\in\ACatu$, there exists a {\htpiso} $\Yon{\cA}\colon\cA\to\rep{\cA}$ with $\rep{\cA}\in\dgCat$.
\end{lem}

\begin{proof}
See \cite[Corollary 1.4]{BLM}.
\end{proof}

Coming to prenatural transformations, we will consider the following notion.

\begin{definition}
Given $\fF,\fG\colon\cA\to\cB$ in $\ACatn$ with $\cA$ strictly unital, a prenatural transformation $\nat\colon\fF\to\fG$ is \emph{strictly unital} if $\nat^i(f_i\otimes\cdots\otimes f_1)=0$ whenever $i>0$ and there exists $j\in\{1,\dots,i\}$ such that $f_j=\id_A$ for some $A\in\cA$.
\end{definition}

\subsection{Category of functors and equivalence relations}\label{subsec:equivfun}

Given $\cA,\cB\in\ACatn$, there is a natural non-unital $A_\infty$ category $\FunAn(\cA,\cB)$, whose set of objects is $\ACatn(\cA,\cB)$ and whose morphisms are prenatural transformations (see \cite[Section 1.4]{COS}). As for the maps $\m{i}=\m[\FunAn(\cA,\cB)]{i}$, for our aims it is enough to know that $\m{1}(\nat)^n$ is given by the left-hand-side of \eqref{eq:natrel}, for every $\fF,\fG\in\ACatn(\cA,\cB)$ and every prenatural transformation $\nat\colon\fF\to\fG$ of degree $p$.

Observe that, if $\cB$ is (strictly) unital or is a dg category, then $\FunAn(\cA,\cB)$ has the same property. When $\cA$ and $\cB$ are strictly unital (respectively unital), the full $A_\infty$ subcategory of $\FunAn(\cA,\cB)$ whose set of objects is $\ACat(\cA,\cB)$ (respectively $\ACatu(\cA,\cB)$) will be denoted by $\FunA(\cA,\cB)$ (respectively $\FunAu(\cA,\cB)$).

\begin{definition}\label{weakeqhom}
Let $\fF,\fG\in\ACatn(\cA,\cB)$.	

(i) $\fF$ and $\fG$ are \emph{weakly equivalent} (denoted by $\fF\hiso\fG$) if $\cB$ is unital and $\fF\iso\fG$ in the (unital) category $H^0\bigl(\FunAn(\cA,\cB)\bigr)$.

(ii) $\fF$ and $\fG$ are \emph{homotopic} (denoted by $\fF\htp\fG$) if $\fF^0=\fG^0$ and there exists a prenatural transformation $\nat\colon\fF\to\fG$ of degree $0$ such that $\nat^0=0$ and $\fG^i=\fF^i+\m{1}(\nat)^i$ for every $i>0$.
\end{definition}

\begin{remark}\label{htpeqrel}
As it is proved in \cite[Section (1h)]{Sei}, homotopy is an equivalence relation. Since it will be useful later, we also point out the following property, which can be directly deduced from the proof. Let $\fF,\fG,\fH\in\ACatn(\cA,\cB)$ with $\fF\htp\fG$ and $\fG\htp\fH$ through homotopies $\nat_1$ and $\nat_2$, respectively. Assuming that there exists $n>0$ such that $\nat_2^i=0$ for $i<n$, then $\fF\htp\fH$ through a homotopy $\nat$ such that $\nat^i=\nat_1^i$ for $i<n$.
\end{remark}

\begin{remark}\label{nathtp}
If $\fF\in\ACatn(\cA,\cB)$ and $\nat\colon\fF\to\fF$ is a prenatural transformation of degree $0$ such that $\nat^0=0$, then there exists $\fG\in\ACatn(\cA,\cB)$ such that $\fF\htp\fG$ with $\fG^i=\fF^i+\m{1}(\nat)^i$ for every $i>0$, where $\nat$ is regarded as a prenatural transformation $\fF\to\fG$. Indeed, we set $\fG^0:=\fF^0$ and, for $n>0$, we define inductively $\fG^n$ as the sum of $\fF^n$ and of the left-hand-side of \eqref{eq:natrel} (which involves $\fG^i$ only with $i<n$, since $\nat^0=0$). Then $\fG^i=\fF^i+\m{1}(\nat)^i$ for $i>0$ by construction, while the fact that $\fG\in\ACatn(\cA,\cB)$ follows from \autoref{natfun}, taking into account that $\m{1}(\nat)$ is a natural transformation (because $\m{1}\comp\m{1}=0$) of degree $1$ and clearly $\m{1}(\nat)^0=0$.
\end{remark}

Since $\hiso$ is compatible with compositions, from the category $\ACatu$ one can obtain a quotient category $\QACatu$ with the same objects and whose morphisms are given by
\[
\QACatu(\cA,\cB):=\ACatu(\cA,\cB)\quot\hiso.
\]
Similarly one can construct $\QACat$ from $\ACat$.

Later we will need the following results.

\begin{lem}\label{nathiso}
Let $\fF,\fF'\in\ACatn(\cA,\cB)$ with $\cB$ unital. If there exists a natural transformation $\nat\colon\fF\to\fF'$ of degree $0$ such that $H(\nat)\colon H(\fF)\to H(\fF')$ is an isomorphism, then $\fF\hiso\fF'$.
\end{lem}

\begin{proof}
It follows from \cite[Proposition 7.15]{Lyu}.
\end{proof}

\begin{lem}\label{hequiv}
Let $\fF\colon\cA\to\cB$ be a {\htpequiv} (in particular, $\cB$ is unital). Then $\cA$ and $\fF$ are also unital and there exists $\fG\in\ACatu(\cB,\cA)$ such that $\fG\comp\fF\hiso\id_\cA$ and $\fF\comp\fG\hiso\id_\cB$ (hence the image of $\fF$ is an isomorphism in $\QACatu$).
\end{lem}

\begin{proof}
It follows from \cite[Theorem 8.8]{Lyu}.
\end{proof}

\begin{cor}\label{hisocrit}
Let $\fF,\fF'\in\ACatn(\cA,\cB)$ with $\cB$ unital. Then $\fF\hiso\fF'$ in each of the following cases.
\begin{enumerate}
\item $\cB$ is strictly unital and $\fF\htp\fF'$.
\item There exists a {\htpequiv} $\fG\colon\cB\to\cB'$ such that $\fG\comp\fF\hiso\fG\comp\fF'$.
\item There exists a {\htpequiv} $\fH\colon\cA'\to\cA$ such that $\fF\comp\fH\hiso\fF'\comp\fH$.
\end{enumerate}
\end{cor}

\begin{proof}
If $\cB$ is strictly unital and $\nat\colon\fF\to\fF'$ is a homotopy, it is straightforward to check that the prenatural transformation $\enat\colon\fF\to\fF'$ defined by $\enat^i:=\nat^i$ for $i>0$ and $\enat^0(A):=\id_{\fF^0(A)}$ for every $A\in\cA$ is a natural transformation (see also the paragraph before \cite[Lemma 2.5]{Sei}). Thus part (1) follows from \autoref{nathiso}, whereas parts (2) and (3) are easy consequences of \autoref{hequiv}.
\end{proof}

\section{The non-unital case}\label{sec:adjUI}

Using the notation introduced in \autoref{subsec:barcobar}, we first state the following result.

\begin{prop}\label{dgAnadj}
There is an adjunction
\[
\adjpair{\VdBn}{\ACatn}{\dgCatn}{\fdgAn},
\]
whose unit is $\ncbi\colon\id_{\ACatn}\to\fdgAn\comp\VdBn$ and whose counit is $\ncb\colon\VdBn\comp\fdgAn=\coB\comp\B\to\id_{\dgCatn}$. Moreover, $\ncbi[\cA]$ (for every $\cA\in\ACatn$) and $\ncb[\cB]$ (for every $\cB\in\dgCatn$) are {\htpiso}s.
\end{prop}

It is easy to see that the same proof of \cite[Proposition 1.22]{COS} can be adapted to work when $\K$ is an arbitrary commutative ring and with {\htpiso} in place of quasi-isomorphism. However, when dealing with the strictly unital case in \autoref{sec:infequiv}, it will be useful to know the crucial proof of the fact that, for every $\cA\in\ACatn$ and every $A,A'\in\cA$,
\begin{equation}\label{eq:mapimp}
\ncbi[\cA]^1\colon\cA(A,A')\to\VdBn(\cA)(A,A')=\coB(\Bi(\cA))(A,A')
\end{equation}
is a homotopy equivalence of complexes. Actually we will prove a more general statement in \autoref{ncbiqiso}, from which we will also deduce a result that will be needed in \autoref{sec:IntHom}. To this aim, we first prove a technical result in \autoref{subsec:useful} and then introduce the morphism which replaces \eqref{eq:mapimp} in a more general setting in \autoref{subsec:themor}.

\begin{remark}\label{dgAuadj}
In view of \autoref{unithtpiso}, the adjunction of \autoref{dgAnadj} restricts to adjunctions
\[
\adjpair{\VdBn}{\ACatc}{\dgCatc}{\fdgAn}, \qquad \adjpair{\VdBn}{\ACatu}{\dgCatu}{\fdgAn}.
\]
\end{remark}

\subsection{A criterion for homotopy equivalence}\label{subsec:useful}

We will need the following general and possibly known result about filtered complexes. We include the proof since we could not find a suitable reference.

\begin{lem}\label{filthequiv}
Let $C$ be a complex of $\K$-modules endowed with an ascending and exhaustive filtration $\filt{n}C$ (with $n\in\NN$) such that $\filt{0}C=0$. Assume that for every $n>1$ the exact sequence of complexes
\[
0\to\filt{n-1}C\to\filt{n}C\to\gr{n}C:=(\filt{n}C)/(\filt{n-1}C)\to0
\]
splits as a sequence of graded modules and the complex $\gr{n}C$ is null-homotopic. Then the inclusion $\filt{1}C\mono C$ is a homotopy equivalence of complexes.
\end{lem}

\begin{proof}
We can assume that as a graded module $C=\bigoplus_{n>0}C_n$ with $\filt{n}C=\bigoplus_{0<m\le n}C_m$ and $\gr{n}C=C_n$. We will write $C_{\le n}$ or $C_{<n+1}$ instead of $\filt{n}C$. For $n>0$ we will denote by $d_{\le n}$ the differential of $C_{\le n}$ (which is the restriction of the differential $d$ of $C$) and by $d_n$ the induced differential on $C_n$. Observe that $d_1=d_{\le 1}$ and
\[
d_{\le n}=
\begin{pmatrix}
d_{<n} & e_n \\
0 & d_n
\end{pmatrix}
\colon C_{\le n}=C_{<n}\oplus C_n\to C_{\le n}=C_{<n}\oplus C_n
\]
for $n>1$, where $e_n\colon C_n\to C_{<n}$ is a degree $1$ map such that
\begin{equation}\label{eq:de}
d_{<n}\comp e_n=-e_n\comp d_n.
\end{equation}
Denoting by $i_n\colon C_1\mono C_{\le n}$ the inclusion, we claim that there exist morphisms of complexes $p_n\colon C_{\le n}\to C_1$ and degree $-1$ maps $h_{\le n}\colon C_{\le n}\to C_{\le n}$ such that
\begin{gather}\label{eq:htpinv1}
\id_{C_1}=p_n\comp i_n, \\
\label{eq:htpinv}
\id_{C_{\le n}}=i_n\comp p_n+d_{\le n}\comp h_{\le n}+h_{\le n}\comp d_{\le n}
\end{gather}
for every $n>0$, and satisfying the compatibility conditions
\begin{equation}\label{eq:compat}
p_n\rest{C_{<n}}=p_{n-1}, \qquad h_{\le n}\rest{C_{<n}}=h_{<n}
\end{equation}
for every $n>1$. Assuming this, one can conclude the proof very easily. Indeed, the maps $p\colon C\to C_1$ and $h\colon C\to C$ such that $p\rest{C_{\le n}}=p_n$ and $h\rest{C_{\le n}}=h_{\le n}$ for every $n>0$ are well defined (and unique), thanks to \eqref{eq:compat}. Moreover, since each $p_n$ is a morphism of complexes and \eqref{eq:htpinv1} and \eqref{eq:htpinv} hold, also $p$ is a morphism of complexes and (denoting by $i\colon C_1\mono C$ the inclusion) we obtain
\[
\id_{C_1}=p\comp i, \qquad \id_C=i\comp p+d\comp h+h\comp d,
\]
thus proving that $i$ is a homotopy equivalence.

So it remains to prove the claim, and to this purpose we proceed by induction on $n$. As we can obviously take $p_1=i_1=\id_{C_1}$ and $h_{\le 1}=0$, we assume that $n>1$ and that the maps $p_m$ and $h_m$ with all the required properties have already been chosen for $0<m<n$. In particular, $p_{n-1}$ is a morphism of complexes and
\begin{gather}\label{eq:htpinva1}
\id_{C_1}=p_{n-1}\comp i_{n-1}, \\
\label{eq:htpinva}
\id_{C_{<n}}=i_{n-1}\comp p_{n-1}+d_{<n}\comp h_{<n}+h_{<n}\comp d_{<n}.
\end{gather}
Moreover, since $C_n$ is null-homotopic, there exists a degree $-1$ map $h_n\colon C_n\to C_n$ such that
\begin{equation}\label{eq:htpinvb}
\id_{C_n}=d_n\comp h_n+h_n\comp d_n.
\end{equation}
Setting
\begin{gather*}
p_n:=
\begin{pmatrix}
p_{n-1} & -p_{n-1}\comp e_n\comp h_n
\end{pmatrix}
\colon C_{\le n}=C_{<n}\oplus C_n\to C_1, \\
h_{\le n}:=
\begin{pmatrix}
h_{<n} & -h_{<n}\comp e_n\comp h_n \\
0 & h_n
\end{pmatrix}
\colon C_{\le n}=C_{<n}\oplus C_n\to C_{\le n}=C_{<n}\oplus C_n,
\end{gather*}
\eqref{eq:compat} is certainly satisfied. Using, beyond the fact that $p_{n-1}$ is a morphism of complexes, \eqref{eq:de} and \eqref{eq:htpinvb}, we obtain also
\begin{multline*}
d_1\comp p_n=d_1\comp
\begin{pmatrix}
p_{n-1} & -p_{n-1}\comp e_n\comp h_n
\end{pmatrix}
=
\begin{pmatrix}
d_1\comp p_{n-1} & -d_1\comp p_{n-1}\comp e_n\comp h_n
\end{pmatrix}
\\
=
\begin{pmatrix}
p_{n-1}\comp d_{<n} & -p_{n-1}\comp d_{<n}\comp e_n\comp h_n
\end{pmatrix}
=
\begin{pmatrix}
p_{n-1}\comp d_{<n} & p_{n-1}\comp e_n\comp d_n\comp h_n
\end{pmatrix}
\\
=
\begin{pmatrix}
p_{n-1}\comp d_{<n} & p_{n-1}\comp e_n\comp(\id_{C_n}-h_n\comp d_n)
\end{pmatrix}
=
\begin{pmatrix}
p_{n-1} & -p_{n-1}\comp e_n\comp h_n
\end{pmatrix}
\comp
\begin{pmatrix}
d_{<n} & e_n \\
0 & d_n
\end{pmatrix}
=p_n\comp d_{\le n},
\end{multline*}
which shows that $p_n$ is a morphism of complexes. Taking into account that $p_n\comp i_n=p_{n-1}\comp i_{n-1}$, \eqref{eq:htpinv1} follows directly from \eqref{eq:htpinva1}. Finally, by \eqref{eq:htpinva}, \eqref{eq:htpinvb} and \eqref{eq:de},
\begin{multline*}
d_{\le n}\comp h_{\le n}+h_{\le n}\comp d_{\le n}=
\begin{pmatrix}
d_{<n} & e_n \\
0 & d_n
\end{pmatrix}
\comp
\begin{pmatrix}
h_{<n} & -h_{<n}\comp e_n\comp h_n \\
0 & h_n
\end{pmatrix}
+
\begin{pmatrix}
h_{<n} & -h_{<n}\comp e_n\comp h_n \\
0 & h_n
\end{pmatrix}
\comp
\begin{pmatrix}
d_{<n} & e_n \\
0 & d_n
\end{pmatrix}
\\
=
\begin{pmatrix}
d_{<n}\comp h_{<n}+h_{<n}\comp d_{<n} & -d_{<n}\comp h_{<n}\comp e_n\comp h_n+e_n\comp h_n+h_{<n}\comp e_n-h_{<n}\comp e_n\comp h_n\comp d_n \\
0 & d_n\comp h_n+h_n\comp d_n
\end{pmatrix}
\\
=
\begin{pmatrix}
\id_{C_{<n}}-i_{n-1}\comp p_{n-1} & (h_{<n}\comp d_{<n}+i_{n-1}\comp p_{n-1})\comp e_n\comp h_n+h_{<n}\comp e_n\comp d_n\comp h_n \\
0 & \id_{C_n}
\end{pmatrix}
\\
=
\begin{pmatrix}
\id_{C_{<n}}-i_{n-1}\comp p_{n-1} & i_{n-1}\comp p_{n-1}\comp e_n\comp h_n \\
0 & \id_{C_n}
\end{pmatrix}
=\id_{C_{\le n}}-i_n\comp p_n,
\end{multline*}
which proves \eqref{eq:htpinv}.
\end{proof}

\begin{remark}\label{filtqiso}
The assumption that the sequence splits in \autoref{filthequiv} is essential. To see this, just consider the case in which $C$ is given by a non-split short exact sequence $0\to M\mor{i}N\to P\to0$ and $\filt{1}C$ is the subcomplex $0\to M\isomor i(M)\to0$, while $\filt{n}C=C$ for $n>1$. On the other hand, one can very easily prove that the inclusion $\filt{1}C\mono C$ is a quasi-isomorphism of complexes, even without that assumption.
\end{remark}

\subsection{The relevant morphism}\label{subsec:themor}

In this section, we fix two non-unital $A_\infty$ categories $\cA$ and $\cB$, and we denote by $\cC$ the non-unital dg cocategory $\red{\aug{\Bi(\cA)}\otimes\aug{\Bi(\cB)}}$. We will also consider the dg quiver $\red{\aug{\cA}\otimes\aug{\cB}}$.

Our aim here is to construct a suitable morphism of complexes
\[
\red{\aug{\cA}\otimes\aug{\cB}}\bigl((A,B),(A',B')\bigr)\to\coB(\cC)\bigl((A,B),(A',B')\bigr)
\]
where $A,A'\in\cA$ and $B,B'\in\cB$. Its precise definition is in \eqref{eq:AB} below. 

Moving in this direction, note that, given $A,A'\in\cA$ and $B,B'\in\cB$, we have
\begin{equation}\label{eq:HomAB}
\red{\aug{\cA}\otimes\aug{\cB}}\bigl((A,B),(A',B')\bigr)=\cA(A,A')\otimes\cB(B,B')\oplus\cA(A,A')^{\delta_{B,B'}}\oplus\cB(B,B')^{\delta_{A,A'}}
\end{equation}
as a dg $\K$-module. In order to explicitly describe also $\coB(\cC)\bigl((A,B),(A',B')\bigr)$, we first introduce some notation.

For every $i,j\in\NN$, we denote by $C_{i,j}\bigl((A,B),(A',B')\bigr)$ the graded $\K$-module
\[
\bigoplus_{\substack{A=A_0,A_1,\dots,A_{i-1},A_i=A'\in\cA\\ B=B_0,B_1,\dots,B_{j-1},B_j=B'\in\cB}}\sh{\cA(A_{i-1},A_i)}\otimes\cdots\otimes\sh{\cA(A_0,A_1)}\otimes\sh{\cB(B_{j-1},B_j)}\otimes\cdots\otimes\sh{\cB(B_0,B_1)},
\]
which is meant to be $0$ when $i=0$ and $A\ne A'$ or $j=0$ and $B\ne B'$ or $i=j=0$. Given $m_1,n_1,\dots,m_l,n_l\in\NN$, we denote by $C_{(m_1,\dots,m_l),(n_1,\dots,n_l)}\bigl((A,B),(A',B')\bigr)$ the graded $\K$-module
\[
\bigoplus_{\substack{A=A_0,A_1,\dots,A_{l-1},A_l=A'\in\cA\\ B=B_0,B_1,\dots,B_{l-1},B_l=B'\in\cB}}\sh[-1]{C_{m_l,n_l}\bigl((A_{l-1},B_{l-1}),(A_l,B_l)\bigr)}\otimes\cdots\otimes\sh[-1]{C_{m_1,n_1}\bigl((A_0,B_0),(A_1,B_1)\bigr)}
\]
(in particular, $C_{(i),(j)}\bigl((A,B),(A',B')\bigr)=\sh[-1]{C_{i,j}\bigl((A,B),(A',B')\bigr)}$). For every $m,n\ge0$ we define moreover
\[
\grVn{m,n}\bigl((A,B),(A',B')\bigr):=\bigoplus_{\substack{m_1+\cdots+m_l=m\\ n_1+\cdots+n_l=n}}C_{(m_1,\dots,m_l),(n_1,\dots,n_l)}\bigl((A,B),(A',B')\bigr).
\]
Note that, in particular, $\grVn{0,0}\bigl((A,B),(A',B')\bigr)=0$ and
\begin{gather}
\label{eq:Hom10}
\grVn{1,0}\bigl((A,B),(A',B')\bigr)=\cA(A,A')^{\delta_{B,B'}} \\
\label{eq:Hom01}
\grVn{0,1}\bigl((A,B),(A',B')\bigr)=\cB(B,B')^{\delta_{A,A'}} \\
\begin{split}\label{eq:Hom11}
\grVn{1,1}\bigl((A,B),(A',B')\bigr)= & \cA(A,A')\otimes\cB(B,B')\oplus\cB(B,B')\otimes\cA(A,A') \\
 & \oplus\sh[-1]{(\sh{\cA(A,A')}\otimes\sh{\cB(B,B')})}
\end{split}
\end{gather}
Then the non-unital dg category $\coB(\cC)$ has the same objects as $\red{\aug{\cA}\otimes\aug{\cB}}$, and
\[
\coB(\cC)\bigl((A,B),(A',B')\bigr)=\bigoplus_{m,n\ge0}\grVn{m,n}\bigl((A,B),(A',B')\bigr)
\]
as a graded $\K$-module for every $A,A'\in\cA$ and $B,B'\in\cB$. While the composition in $\coB(\cC)$ is the natural one given by the tensor product of the cobar construction, the differential on $\coB(\cC)\bigl((A,B),(A',B')\bigr)$ extends (in such a way that the graded Leibnitz rule holds) $\cod+\com$, where $\cod$ and $\com$ are determined, respectively, by the differential and the comultiplication on the dg cocategory $\cC$. More precisely, given
\[
c=\sh[-1]{(\sh{f_m}\otimes\cdots\otimes\sh{f_1}\otimes\sh{g_n}\otimes\cdots\otimes\sh{g_1})}\in C_{(m),(n)}\bigl((A,B),(A',B')\bigr)
\]
with the $f_i$ and the $g_j$ homogeneous, we have
\[
\com(c)=\sum_{(i,j)\in I_{m,n}}(-1)^{\deg(c_{\le i,\emptyset})\degp(c_{\emptyset,>j})+\deg(c_{>i,\emptyset})}c_{>i,>j}\otimes c_{\le i,\le j},
\]
where $I_{m,n}:=\{0,\dots,m\}\times\{0,\dots,n\}\setminus\{(0,0),(m,n)\}$. Here $>i$ and $\le i$ (respectively $>j$ and $\le j$) denote the (descending) intervals $[m,i)=[m.i+1]$ and $[i,1]$ (respectively $[n,j)$ and $[j,1]$), and in general
\[
c_{[i',i),[j',j)}:=\sh[-1]{(\sh{f_{i'}}\otimes\cdots\otimes\sh{f_{i+1}}\otimes\sh{g_{j'}}\otimes\cdots\otimes\sh{g_{j+1}})}
\]
for $1\le i\le i'\le m$ and $1\le j\le j'\le n$. Obviously the empty interval is denoted also by $\emptyset$, while the full interval $[m,1]$ or $[n,1]$ will be denoted by $*$. Clearly
\[
\com\Bigl(\grVn{m,n}\bigl((A,B),(A',B')\bigr)\Bigr)\subseteq\grVn{m,n}\bigl((A,B),(A',B')\bigr)
\]
for every $m,n\ge0$. On the other hand, the component $\cod^1$ of $\cod$ induced from $\m[\cA]{1}$ and $\m[\cB]{1}$ is given on $c$ as above by 
\[
\cod^1(c)=\sum_{i=1}^m(-1)^{\degp(c_{>i,\emptyset})}\cod^1_{i,0}(c)+\sum_{j=1}^n(-1)^{\degp(c_{*,>j})}\cod^1_{0,j}(c),
\]
where
\begin{gather*}
\cod^1_{i,0}(c):=\sh[-1]{(\sh{f_m}\otimes\cdots\otimes\sh{f_{i+1}}\otimes\sh{\m[\cA]{1}(f_i)}\otimes\sh{f_{i-1}}\otimes\cdots\otimes\sh{f_1}\otimes\sh{g_n}\otimes\cdots\otimes\sh{g_1})}, \\
\cod^1_{0,j}(c):=\sh[-1]{(\sh{f_m}\otimes\cdots\otimes\sh{f_1}\otimes\sh{g_n}\otimes\cdots\otimes\sh{g_{j+1}}\otimes\sh{\m[\cB]{1}(g_j)}\otimes\sh{g_{j-1}}\otimes\cdots\otimes\sh{g_1})}.
\end{gather*}
Hence also in this case
\[
\cod^1\Bigl(\grVn{m,n}\bigl((A,B),(A',B')\bigr)\Bigr)\subseteq\grVn{m,n}\bigl((A,B),(A',B')\bigr)
\]
for every $m,n\ge0$. As for the other components $\cod^i$ of $\cod$ induced from $\m[\cA]{i}$ and $\m[\cB]{i}$ with $i>1$, for our purposes it is enough to observe that
\[
\cod^i\Bigl(\grVn{m,n}\bigl((A,B),(A',B')\bigr)\Bigr)\subseteq\bigoplus_{0<m'<m}\grVn{m',n}\bigl((A,B),(A',B')\bigr)\bigoplus_{0<n'<n}\grVn{m,n'}\bigl((A,B),(A',B')\bigr)
\]
for every $m,n\ge0$. This implies that $\cod^1+\com$ is a differential (often denoted simply by $d$) on each $\grVn{m,n}\bigl((A,B),(A',B')\bigr)$, which will be regarded as a complex in this way. Moreover,
\begin{gather*}
\grVn{*,0}\bigl((A,B),(A',B')\bigr):=\bigoplus_{m\ge0}\grVn{m,0}\bigl((A,B),(A',B')\bigr) \\
\grVn{0,*}\bigl((A,B),(A',B')\bigr):=\bigoplus_{n\ge0}\grVn{0,n}\bigl((A,B),(A',B')\bigr) \\
\grVn{>0}\bigl((A,B),(A',B')\bigr):=\bigoplus_{m,n>0}\grVn{m,n}\bigl((A,B),(A',B')\bigr)
\end{gather*}
are subcomplexes of $\coB(\cC)\bigl((A,B),(A',B')\bigr)$, and obviously there is a decomposition
\[
\coB(\cC)\bigl((A,B),(A',B')\bigr)=\grVn{*,0}\bigl((A,B),(A',B')\bigr)\oplus\grVn{0,*}\bigl((A,B),(A',B')\bigr)\oplus\grVn{>0}\bigl((A,B),(A',B')\bigr).
\]

Now, for every $A,A'\in\cA$ and $B,B'\in\cB$, we will consider the maps (see \eqref{eq:Hom11})
\begin{equation}\label{eq:inc11}
\begin{split}
\cA(A,A')\otimes\cB(B,B') & \to\grVn{1,1}\bigl((A,B),(A',B')\bigr) \\
f\otimes g & \mapsto(0,(-1)^{\deg(f)\deg(g)}g\otimes f,0)
\end{split}
\end{equation}
and
\begin{equation}\label{eq:pro11}
\begin{split}
\grVn{1,1}\bigl((A,B),(A',B')\bigr) & \to\cA(A,A')\otimes\cB(B,B') \\
\bigl(f\otimes g,g'\otimes f',\sh[-1]{(\sh{f''}\otimes\sh{g''})}\bigr) & \mapsto f\otimes g+(-1)^{\deg(f')\deg(g')}f'\otimes g'
\end{split}
\end{equation}
It is easy to check that \eqref{eq:inc11} and \eqref{eq:pro11} are morphisms of complexes. Remembering \eqref{eq:HomAB}, it is then clear that \eqref{eq:Hom10}, \eqref{eq:Hom01} and \eqref{eq:inc11} define the morphism of complexes
\begin{equation}\label{eq:AB}
\red{\aug{\cA}\otimes\aug{\cB}}\bigl((A,B),(A',B')\bigr)\to\coB(\cC)\bigl((A,B),(A',B')\bigr)=\bigoplus_{m,n\ge0}\grVn{m,n}\bigl((A,B),(A',B')\bigr)
\end{equation}
we are interested in. 

\subsection{A general result}\label{ncbiqiso}

In order to prove the properties of \eqref{eq:AB} we first need the following key result, whose proof is rather technical.

\begin{lem}\label{null}
Let $A,A'\in\cA$ and $B,B'\in\cB$.
\begin{enumerate}
\item The maps \eqref{eq:inc11} and \eqref{eq:pro11} are mutually inverse homotopy equivalences of complexes.
\item Given $m,n\in\NN$ with $m>1$ or $n>1$, the complex $\grVn{m,n}\bigl((A,B),(A',B')\bigr)$ is null-homotopic.
\end{enumerate}
\end{lem}

\begin{proof}
We start by observing that the composition of \eqref{eq:pro11} with \eqref{eq:inc11} is $\id_{\cA(A,A')\otimes\cB(B,B')}$, while the composition of \eqref{eq:inc11} with \eqref{eq:pro11} is the map
\[
\begin{split}
\pro\colon\grVn{1,1}\bigl((A,B),(A',B')\bigr) & \to\grVn{1,1}\bigl((A,B),(A',B')\bigr) \\
\bigl(f\otimes g,g'\otimes f',\sh[-1]{(\sh{f''}\otimes\sh{g''})}\bigr) & \mapsto\bigl(0,(-1)^{\deg(f)\deg(g)}g\otimes f+g'\otimes f',0\bigr)
\end{split}
\]
We define also $\pro\colon\grVn{m,n}\bigl((A,B),(A',B')\bigr)\to\grVn{m,n}\bigl((A,B),(A',B')\bigr)$ to be $0$ for $m>1$, and to be the identity for $m=1$ and $n=0$.

Therefore, in order to prove both (1) and (2) (where, by symmetry, we can assume $m>1$), we just need to find, when $m>1$ or $m=n=1$, a $\K$-linear map
\[
r\colon\grVn{m,n}\bigl((A,B),(A',B')\bigr)\to\grVn{m,n}\bigl((A,B),(A',B')\bigr)
\]
of degree $-1$ such that
\[
d\comp r+r\comp d=\id-\pro.
\]
More generally, we define $r$ for $m>0$ as follows. By linearity an element of $\grVn{m,n}\bigl((A,B),(A',B')\bigr)$ can be assumed to be of the form
\[
c=c^l\otimes\cdots\otimes c^1\in C_{(m_1,\dots,m_l),(n_1,\dots,n_l)}\bigl((A,B),(A',B')\bigr),
\]
where $m_1+\cdots+m_l=m$, $n_1+\cdots+n_l=n$ and $c^k\in C_{(m_k),(n_k)}\bigl((A_{k-1},B_{k-1}),(A_k,B_k)\bigr)$ homogeneous (for $k=1,\dots,l$), with $A_0=A$, $A_l=A'$, $B_0=B$ and $B_l=B'$. Given $1\le i\le j\le l$, we will often use the shorthand $c^{[j,i]}:=c^j\otimes\cdots\otimes c^i$, as well as its variants $c^{[j,i)}$, $c^{(j,i]}$ and $c^{(j,i)}$ (with obvious meanings). Setting
\begin{gather*}
t:=\max\bigl\{k\in\{1,\dots,l\}\st m_k>0\bigr\} \\
s':=\min\bigl\{k\in\{1,\dots,t\}\st m_i=0\text{ for $k<i<t$}\bigr\} \\
s:=
\begin{cases}
s' & \text{if $(m_t,n_t)=(1,0)$} \\
t & \text{otherwise}
\end{cases}
\end{gather*}
(note that they are well defined because $m>0$), we define recursively
\[
r(c):=\sum_{k=s}^{t-1}(-1)^{\deg(c^{[l,t)})+\degp(c^t)\deg(c^{(t,k)})}c^{[l,t)}\otimes c^{(t,k)}\otimes r(c^t\otimes c^k)\otimes c^{(k,1]}
\]
(so $r(c)=0$ if $s=t$, in particular if $(m_t,n_t)\ne(1,0)$). If $(m_t,n_t)=(1,0)$, $c^t=f\in\cA(A_{t-1},A_t)$ and $c^k=\sh[-1]{(\sh{f_{m_k}}\otimes\cdots\otimes\sh{f_1}\otimes\sh{g_{n_k}}\otimes\cdots\otimes\sh{g_1})}$ (with $s\le k<t$), then
\[
r(c^t\otimes c^k):=(-1)^{\deg(f)}\sh[-1]{(\sh{f}\otimes\sh{f_{m_k}}\otimes\cdots\otimes\sh{f_1}\otimes\sh{g_{n_k}}\otimes\cdots\otimes\sh{g_1})}.
\]
For the rest of the proof we assume $m>1$ or $m=n=1$. First we note that $\pro(c)=c^{[l,t)}\otimes\pro(c^{[t,1]})$. Indeed, we can assume $t<l$, and then $\pro(c)=\pro(c^{[t,1]})=0$ if $m>1$, whereas $\pro(c)=c$ and $\pro(c^{[t,1]})=c^{[t,1]}$ if $m=n=1$ (in which case $l=2$, $(m_2,n_2)=(0,1)$ and $(m_1,n_1)=(1,0)$). Since moreover
\begin{multline*}
(d\comp r+r\comp d)(c)=d\bigl((-1)^{\deg(c^{[l,t)})}c^{[l,t)}\otimes r(c^{[t,1]})\bigr)+r\bigl(d(c^{[l,t)})\otimes c^{[t,1]}+(-1)^{\deg(c^{[l,t)})}c^{[l,t)}\otimes d(c^{[t,1]})\bigr) \\
=(-1)^{\deg(c^{[l,t)})}d(c^{[l,t)})\otimes r(c^{[t,1]})+c^{[l,t)}\otimes d\bigl(r(c^{[t,1]})\bigr)+(-1)^{\degp(c^{[l,t)})}d(c^{[l,t)})\otimes r(c^{[t,1]})+c^{[l,t)}\otimes r\bigl(d(c^{[t,1]})\bigr) \\
=c^{[l,t)}\otimes(d\comp r+r\comp d)(c^{[t,1]}),
\end{multline*}
it is enough to prove that
\begin{equation}\label{eq:null}
(d\comp r+r\comp d)(c^{[t,1]})=c^{[t,1]}-\pro(c^{[t,1]}).
\end{equation}
We have
\begin{multline*}
d\bigl(r(c^{[t,1]})\bigr)=d\bigl(\sum_{k=s}^{t-1}(-1)^{\degp(c^t)\deg(c^{(t,k)})}c^{(t,k)}\otimes r(c^t\otimes c^k)\otimes c^{(k,1]}\bigr) \\
=\sum_{k=s}^{t-1}(-1)^{\degp(c^t)\deg(c^{(t,k)})+\deg(c^{(t,k)})}c^{(t,k)}\otimes d\bigl(r(c^t\otimes c^k)\bigr)\otimes c^{(k,1]} \\
+\sum_{k=s}^{t-1}\sum_{i=1}^{k-1}(-1)^{\degp(c^t)\deg(c^{(t,k)})+\degp(c^{[t,i)})}c^{(t,k)}\otimes r(c^t\otimes c^k)\otimes c^{(k,i)}\otimes d(c^i)\otimes c^{(i,1]} \\
+\sum_{k=s}^{t-1}\sum_{i=k+1}^{t-1}(-1)^{\degp(c^t)\deg(c^{(t,k)})+\deg(c^{(t,i)})}c^{(t,i)}\otimes d(c^i)\otimes c^{(i,k)}\otimes r(c^t\otimes c^k)\otimes c^{(k,1]}
\end{multline*}
and
\begin{multline*}
r\bigl(d(c^{[t,1]})\bigr)=r\bigl(\sum_{i=1}^t(-1)^{\deg(c^{[t,i)})}c^{[t,i)}\otimes d(c^i)\otimes c^{(i,1]}\bigr) \\
=r\bigl(d(c^t)\otimes c^{(t,1]}\bigr)+\sum_{i=s}^{t-1}(-1)^{\deg(c^{[t,i)})+\degp(c^t)\deg(c^{(t,i)})}c^{(t,i)}\otimes r\bigl(c^t\otimes d(c^i)\bigr)\otimes c^{(i,1]} \\
+\sum_{i=1}^{t-1}\sum_{k=\max\{i+1,s\}}^{t-1}(-1)^{\deg(c^{[t,i)})+\degp(c^t)\deg(c^{(t,k)})}c^{(t,k)}\otimes r(c^t\otimes c^k)\otimes c^{(k,i)}\otimes d(c^i)\otimes c^{(i,1]} \\
+\sum_{i=1}^{t-1}\sum_{k=s}^{i-1}(-1)^{\deg(c^{[t,i)})+\degp(c^t)\degp(c^{(t,k)})}c^{(t,i)}\otimes d(c^i)\otimes c^{(i,k)}\otimes r(c^t\otimes c^k)\otimes c^{(k,1]},
\end{multline*}
whence
\begin{multline}\label{eq:drrd}
(d\comp r+r\comp d)(c^{[t,1]}) \\
=r\bigl(d(c^t)\otimes c^{(t,1]}\bigr)+\sum_{k=s}^{t-1}(-1)^{\deg(c^t)\deg(c^{(t,k)})}c^{(t,k)}\otimes\Bigl(d\bigl(r(c^t\otimes c^k)\bigr)+(-1)^{\deg(c^t)}r\bigl(c^t\otimes d(c^k)\bigr)\Bigr)\otimes c^{(k,1]}.
\end{multline}
First we assume $(m_t,n_t)\ne(1,0)$, in which case the right-hand-side of \eqref{eq:drrd} is just $r\bigl(\com(c^t)\otimes c^{(t,1]}\bigr)$. If $m_t>1$ then
\begin{multline*}
(d\comp r+r\comp d)(c^{[t,1]})=r\bigl((-1)^{\deg(c^t_{m_t,\emptyset})}c^t_{m_t,\emptyset}\otimes c^t_{<m_t,*}\otimes c^{(t,1]}\bigr) \\
=(-1)^{\deg(c^t_{m_t,\emptyset})}r(c^t_{m_t,\emptyset}\otimes c^t_{<m_t,*})\otimes c^{(t,1]}=c^t\otimes c^{(t,1]}=c^{[t,1]},
\end{multline*}
hence \eqref{eq:null} holds in this case. If $m_t=1$ and $n_t>0$ then
\begin{multline*}
(d\comp r+r\comp d)(c^{[t,1]})=r\Bigl(\bigl((-1)^{\deg(c^t_{1,\emptyset})}c^t_{1,\emptyset}\otimes c^t_{\emptyset,*}-(-1)^{\deg(c^t_{1,\emptyset})\degp(c^t_{\emptyset,*})}c^t_{\emptyset,*}\otimes c^t_{1,\emptyset}\bigr)\otimes c^{(t,1]}\Bigr) \\
=(-1)^{\deg(c^t_{1,\emptyset})}r(c^t_{1,\emptyset}\otimes c^t_{\emptyset,*})\otimes c^{(t,1]} \\
+(-1)^{\deg(c^t_{1,\emptyset})}\sum_{k=s'}^{t-1}(-1)^{\degp(c^t_{1,\emptyset})\bigl(\deg(c^t_{\emptyset,*})+\deg(c^{(t,k)})\bigr)}c^t_{\emptyset,*}\otimes c^{(t,k)}\otimes r(c^t_{1,\emptyset}\otimes c^k)\otimes c^{(k,1]} \\
-(-1)^{\deg(c^t_{1,\emptyset})\degp(c^t_{\emptyset,*})}\sum_{k=s'}^{t-1}(-1)^{\deg(c^t_{\emptyset,*})+\degp(c^t_{1,\emptyset})\deg(c^{(t,k)})}c^t_{\emptyset,*}\otimes c^{(t,k)}\otimes r(c^t_{1,\emptyset}\otimes c^k)\otimes c^{(k,1]}=c^{[t,1]},
\end{multline*}
thus proving \eqref{eq:null} also in this case.

Finally we assume $(m_t,n_t)=(1,0)$. Then we have
\[
r\bigl(d(c^t)\otimes c^{(t,1]}\bigr)=\sum_{k=s}^{t-1}(-1)^{\deg(c^t)\deg(c^{(t,k)})}c^{(t,k)}\otimes r\bigl(\m[\cA]{1}(c^t)\otimes c^k\bigr)\otimes c^{(k,1]},
\]
and so from \eqref{eq:drrd} we obtain
\begin{multline}\label{eq:drrd1}
(d\comp r+r\comp d)(c^{[t,1]}) \\
=\sum_{k=s}^{t-1}(-1)^{\deg(c^t)\deg(c^{(t,k)})}c^{(t,k)}\otimes\Bigl(r\bigl(\m[\cA]{1}(c^t)\otimes c^k\bigr)+d\bigl(r(c^t\otimes c^k)\bigr)+(-1)^{\deg(c^t)}r\bigl(c^t\otimes d(c^k)\bigr)\Bigr)\otimes c^{(k,1]}.
\end{multline}
Since
\begin{multline*}
\cod^1\bigl(r(c^t\otimes c^k)\bigr)=\sum_{i=1}^{m_k+1}(-1)^{\degp\bigl(r(c^t\otimes c^k)_{>i,\emptyset}\bigr)}\cod^1_{i,0}\bigl(r(c^t\otimes c^k)\bigr)+\sum_{j=1}^{n_k}(-1)^{\degp\bigl(r(c^t\otimes c^k)_{*,>j}\bigr)}\cod^1_{0,j}\bigl(r(c^t\otimes c^k)\bigr) \\
=\sum_{i=1}^{m_k}(-1)^{\deg(c^t)+\deg(c^k_{>i,\emptyset})}r\bigl(c^t\otimes\cod^1_{i,0}(c^k)\bigr)-r\bigl(\m[\cA]{1}(c^t)\otimes c^k\bigr)+\sum_{j=1}^{n_k}(-1)^{\deg(c^t)+\deg(c^k_{*,>j})}r\bigl(c^t\otimes\cod^1_{0,j}(c^k)\bigr) \\
=-r\bigl(\m[\cA]{1}(c^t)\otimes c^k\bigr)-(-1)^{\deg(c^t)}r\Bigl(c^t\otimes\bigl(\sum_{i=1}^{m_k}(-1)^{\degp(c^k_{>i,\emptyset})}\cod^1_{i,0}(c^k)+\sum_{j=1}^{n_k}(-1)^{\degp(c^k_{*,>j})}\cod^1_{0,j}(c^k)\bigr)\Bigr) \\
=-r\bigl(\m[\cA]{1}(c^t)\otimes c^k\bigr)-(-1)^{\deg(c^t)}r\bigl(c^t\otimes\cod^1(c^k)\bigr)
\end{multline*}
and
\begin{multline*}
\com\bigl(r(c^t\otimes c^k)\bigr) \\
=\sum_{(i,j)\in I_{m_k+1,n_k}}(-1)^{\deg\bigl(r(c^t\otimes c^k)_{\le i,\emptyset}\bigr)\degp\bigl(r(c^t\otimes c^k)_{\emptyset,>j}\bigr)+\deg\bigl(r(c^t\otimes c^k)_{>i,\emptyset}\bigr)}r(c^t\otimes c^k)_{>i,>j}\otimes r(c^t\otimes c^k)_{\le i,\le j} \\
=\sum_{(i,j)\in I_{m_k,n_k}}(-1)^{\deg(c^k_{\le i,\emptyset})\degp(c^k_{\emptyset,>j})+\deg(c^k_{>i,\emptyset})+\degp(c^t)}r(c^t\otimes c^k_{>i,>j})\otimes c^k_{\le i,\le j}+c^t\otimes c^k \\
-(-1)^{\deg(c^t)\deg(c^k)}(c^k\otimes c^t)^{\delta_{m_k,0}}+\sum_{j=\delta_{m_k,0}}^{n_k-1}(-1)^{\bigl(\deg(c^k_{*,\emptyset})+\degp(c^t)\bigr)\degp(c^k_{\emptyset,>j})+1}c^k_{\emptyset,>j}\otimes r(c^t\otimes c^k_{*,\le j}) \\
=c^t\otimes c^k-(-1)^{\deg(c^t)\deg(c^k)}(c^k\otimes c^t)^{\delta_{m_k,0}} \\
-(-1)^{\deg(c^t)}r\bigl(c^t\otimes\sum_{(i,j)\in I_{m_k,n_k}}(-1)^{\deg(c^k_{\le i,\emptyset})\degp(c^k_{\emptyset,>j})+\deg(c^k_{>i,\emptyset})}c^k_{>i,>j}\otimes c^k_{\le i,\le j}\bigr) \\
=c^t\otimes c^k-(-1)^{\deg(c^t)\deg(c^k)}(c^k\otimes c^t)^{\delta_{m_k,0}}-(-1)^{\deg(c^t)}r\bigl(c^t\otimes\com(c^k)\bigr),
\end{multline*}
we see that
\[
r\bigl(\m[\cA]{1}(c^t)\otimes c^k\bigr)+d\bigl(r(c^t\otimes c^k)\bigr)+(-1)^{\deg(c^t)}r\bigl(c^t\otimes d(c^k)\bigr)=c^t\otimes c^k-(-1)^{\deg(c^t)\deg(c^k)}(c^k\otimes c^t)^{\delta_{m_k,0}}.
\]
Substituting the last equality in \eqref{eq:drrd1} and remembering that $m_k=0$ for $s<k<t$, while $m_s=0$ if and only if $m=1$ (in which case $s=1$), we get
\[
(d\comp r+r\comp d)(c^{[t,1]})=
\begin{cases}
c^{[t,1]} & \text{if $m>1$} \\
c^{[t,1]}-(-1)^{\deg(c^t)\deg(c^{(t,1]})}c^{(t,1]}\otimes c^t & \text{if $m=1$},
\end{cases}
\]
from which we conclude that \eqref{eq:null} is satisfied also in this case.
\end{proof}

We can finally prove the main result of this section. Note that, when $\cB$ is the $0$ dg algebra, \eqref{eq:AB} boils down to \eqref{eq:mapimp}. Hence the first part of the following result shows, in particular, that \eqref{eq:mapimp} is a homotopy equivalence, as wanted.

\begin{prop}\label{hiso}
For every $A,A'\in\cA$ and $B,B'\in\cB$ the map \eqref{eq:AB} is a homotopy equivalence. Moreover, a morphism of complexes
\[
\coB(\cC)\bigl((A,B),(A',B')\bigr)\to\red{\aug{\cA}\otimes\aug{\cB}}\bigl((A,B),(A',B')\bigr)
\]
is a homotopy equivalence if and only if its restriction to $\bigoplus_{0\le m,n\le1}\grVn{m,n}\bigl((A,B),(A',B')\bigr)$ is a homotopy equivalence.
\end{prop}

\begin{proof}
By construction \eqref{eq:AB} is the composition of a map
\begin{equation}\label{eq:AB1}
\red{\aug{\cA}\otimes\aug{\cB}}\bigl((A,B),(A',B')\bigr)\to\bigoplus_{0\le m,n\le1}\grVn{m,n}\bigl((A,B),(A',B')\bigr)
\end{equation}
with the inclusion
\begin{equation}\label{eq:inc}
\bigoplus_{0\le m,n\le1}\grVn{m,n}\bigl((A,B),(A',B')\bigr)\mono\coB(\cC)\bigl((A,B),(A',B')\bigr).
\end{equation}
Now, part (1) of \autoref{null} immediately implies that \eqref{eq:AB1} is a homotopy equivalence. Therefore it is enough to prove that \eqref{eq:inc} is a homotopy equivalence, as well. Clearly this is true if (and only if) each of the three inclusions
\begin{gather*}
\grVn{1,0}\bigl((A,B),(A',B')\bigr)\mono\grVn{*,0}\bigl((A,B),(A',B')\bigr) \\
\grVn{0,1}\bigl((A,B),(A',B')\bigr)\mono\grVn{0,*}\bigl((A,B),(A',B')\bigr) \\
\grVn{1,1}\bigl((A,B),(A',B')\bigr)\mono\grVn{>0}\bigl((A,B),(A',B')\bigr)
\end{gather*}
is a homotopy equivalence. To this aim we apply \autoref{filthequiv} to the complexes on the right-hand-sides of the above inclusions, endowed with the filtrations
\begin{gather*}
\filt{n}\grVn{*,0}\bigl((A,B),(A',B')\bigr):=\bigoplus_{0\le m\le n}\grVn{m,0}\bigl((A,B),(A',B')\bigr) \\
\filt{n}\grVn{0,*}\bigl((A,B),(A',B')\bigr):=\bigoplus_{0\le m\le n}\grVn{0,m}\bigl((A,B),(A',B')\bigr) \\
\filt{n}\grVn{>0}\bigl((A,B),(A',B')\bigr):=\bigoplus_{m,m'>0,\,m+m'\le n+1}\grVn{m,m'}\bigl((A,B),(A',B')\bigr)
\end{gather*}
Note that the assumptions of \autoref{filthequiv} are satisfied because
\begin{gather*}
\gr{n}\grVn{*,0}\bigl((A,B),(A',B')\bigr)=\grVn{n,0}\bigl((A,B),(A',B')\bigr) \\
\gr{n}\grVn{0,*}\bigl((A,B),(A',B')\bigr)=\grVn{0,n}\bigl((A,B),(A',B')\bigr) \\
\gr{n}\grVn{>0}\bigl((A,B),(A',B')\bigr)=\bigoplus_{0<m\le n}\grVn{m,n+1-m}\bigl((A,B),(A',B')\bigr)
\end{gather*}
are null-homotopic for $n>1$ by part (2) of \autoref{null}.
\end{proof}

Let us now single out the following direct consequence.

\begin{cor}\label{hiso2}
Assume that $\cA$ and $\cB$ are non-unital dg categories. Then there is a natural non-unital dg functor $\dgmf{\fN}\colon\coB(\cC)\to\red{\aug{\cA}\otimes\aug{\cB}}$, which is a \htpiso.
\end{cor}

\begin{proof}
The definition of $\dgmf{\fN}$ can be found in \cite[Section 3.1]{COS}, where it is also proved that it is a quasi-isomorphism (hence a \htpiso) when $\K$ is a field. Over an arbitrary commutative ring we can apply the second part of \autoref{hiso}. Indeed, it can be readily checked that, for every $A,A'\in\cA$ and $B,B'\in\cB$, the restriction of
\[
\dgmf{\fN}\colon\coB(\cC)\bigl((A,B),(A',B')\bigr)\to\red{\aug{\cA}\otimes\aug{\cB}}\bigl((A,B),(A',B')\bigr)
\]
to $\bigoplus_{0\le m,n\le1}\grVn{m,n}\bigl((A,B),(A',B')\bigr)$ is given by the natural maps \eqref{eq:Hom10}, \eqref{eq:Hom01} and \eqref{eq:pro11}. Such restriction is then a homotopy equivalence by part (1) of \autoref{null}.
\end{proof}

\section{Proof of \autoref{infequiv}}\label{sec:infequiv}

We refer to \cite{Tan} and \cite{Pas} for the (few) basic notions about $\infty$-categories which are needed in this section.

Recall that a \emph{category with weak equivalences} is a pair $(\cC,\W)$ (often denoted simply by $\cC$), where $\cC$ is an ordinary category and $\W$ is a class of morphisms of $\cC$ containing the isomorphisms and satisfying the 2-out-of-3 property. We can then consider both the localization $\Ho\cC$ of $\cC$ with respect to $\W$, which is an ordinary category,\footnote{In general it can be necessary to enlarge the universe of definition in order to get a category, because the morphisms in $\Ho\cC$ between two objects need not be a set in the original universe. However a posteriori this does not happen in the cases we are interested in, namely when $\cC$ is one of the four categories in \eqref{eq:incl} and $\W$ is the corresponding class of quasi-equivalences. Indeed, this follows from \autoref{1equiv} and from the fact that $\Hqe$ can be alternatively described as the homotopy category of $\dgCat$ endowed with Tabuada's model structure \cite{Ta}.} and the $\infty$-category $\Inf\cC$ obtained as the localization of the nerve of $\cC$ with respect to the image under the nerve functor of $\W$.

Given another category with weak equivalences $\cD$, a functor $\fF\colon\cC\to\cD$ preserving weak equivalences induces a functor $\Ho\fF\colon\Ho\cC\to\Ho\cD$ and an $\infty$-functor $\Inf\fF\colon\Inf\cC\to\Inf\cD$. Now, let $\fF'\colon\cC\to\cD$ be another functor such that there exists a natural transformation $\fF\to\fF'$ or $\fF'\to\fF$ whose components are all weak equivalences. Then it is easy to see (using the 2-out-of-3 property) that also $\fF'$ preserves weak equivalences and that $\Ho\fF$ is isomorphic to $\Ho{\fF'}$. Moreover, by \cite[Proposition 2.3]{Tan}, $\Inf\fF$ is homotopic to $\Inf{\fF'}$. Clearly this implies the following result.

\begin{prop}\label{Tan}
Let $\cC$ and $\cD$ be categories with weak equivalences and let $\fF\colon\cC\to\cD$ and $\fG\colon\cD\to\cC$ be functors preserving weak equivalences. Assume moreover that $\id_{\cC}$ (respectively $\id_{\cD}$) and $\fG\comp\fF$ (respectively $\fF\comp\fG$) are related by a zigzag of natural transformations whose components are all weak equivalences. Then $\Ho\fF$ and $\Ho\fG$ are quasi-inverse equivalences of categories and $\Inf\fF$ and $\Inf\fG$ are quasi-inverse equivalences of $\infty$-categories.
\end{prop}

\begin{remark}\label{Pas}
The assumptions of \autoref{Tan} are certainly satisfied when $\fF$ and $\fG$ form a \emph{Dwyer-Kan adjunction}, meaning that they are adjoint functors which preserve weak equivalences and such that the components of both the unit and the counit of the adjunction are all weak equivalences. The fact that in this case $\Inf\fF$ and $\Inf\fG$ are quasi-inverse equivalences of $\infty$-categories was already observed in \cite{Pas}.
\end{remark}

This discussion is crucial for our proof of \autoref{infequiv} which is divided into several sections each of which compares the various $\infty$-categorical localizations mentioned in the statement.

\begin{remark}\label{rmk:algebras}
The reader should approach the next few sections by paying attention to the fact that all functors involved in our constructions induce bijections of objects. Thus our argument provides a proof of the variant of \autoref{infequiv} where categories are replaced by algebras (namely, categories with just one object). Clearly this will prove all the equivalences in \autoref{infalg}, except the last one, which can however be very easily deduced from \autoref{dgAnadj} (see also \autoref{adjunit}).
\end{remark}

\subsection{Dg vs.\  strictly unital $A_\infty$ categories}\label{subsec:Pas}

In this subsection we prove the equivalence between $\InfdgCat$ and $\InfACat$. As a preliminary step, let us consider the following.

\begin{definition}\label{def:splitunit}
A strictly unital $A_\infty$ category $A$ \emph{has split units} if $\K\iso\K\id_A$ and the inclusion $\K\id_A\mono\cA(A,A)^0$ of $\K$-modules splits, for every $A\in\cA$.
\end{definition}

In view of \autoref{Pas}, the result we want to prove is a straightforward consequence of the following one, which is the strictly unital analogue of \autoref{dgAnadj} in the non-unital setting.

\begin{prop}\label{dgAadj}
There is an adjunction
\[
\adjpair{\VdB}{\ACat}{\dgCat}{\fdgA},
\]
where $\fdgA$ is the inclusion functor. Moreover, the unit $\ntV\colon\id_{\ACat}\to\fdgA\comp\VdB$ and the counit $\nfV\colon\VdB\comp\fdgA\to\id_{\dgCat}$ are such that $\ntV[\cA]$ (for every $\cA\in\ACat$) and $\nfV[\cB]$ (for every $\cB\in\dgCat$) are quasi-isomorphisms. Finally, $\ntV[\cA]$ is even a {\htpiso} if $\cA$ has split units.
\end{prop}

Before turning to the proof of \autoref{dgAadj}, a few observations are in order.

\begin{remark}\label{adjunit}
In a completely similar way, replacing \autoref{dgAadj} with \autoref{dgAnadj} and recalling \autoref{dgAuadj}, the adjunction between $\VdBn$ and $\fdgAn$ induces a natural equivalence of $\infty$-categories between $\InfdgCatu$ (respectively $\InfdgCatc$) and $\InfACatu$ (respectively $\InfACatc$).
\end{remark}

\begin{remark}\label{varInfCat}
When $\K$ is a field, the equivalence between $\InfdgCat$ and $\InfACat$ was already proved in \cite[Corollary 5.2]{Pas} (using the results of \cite{COS}). As it is pointed out in \cite{Pas}, one can consider different models for $\infty$-categories and for all of them there is an analogue of \cite[Corollary 5.2]{Pas}. Namely, one gets  \cite[Corollaries 2.5, 3.2, 4.3, 4.5, 5.1]{Pas}. Using \autoref{dgAadj} one sees that they remain valid when $\K$ is an arbitrary commutative ring.
\end{remark}

\begin{remark}\label{pretr}
Actually we can say that
\[
\adjpair{\VdB}{\ACat}{\dgCat}{\fdgA}
\]
is a Dwyer-Kan adjunction even if we take as weak equivalences, instead of the quasi-equivalences, the pretriangulated (or Morita) equivalences (see \cite[\S 1.4 and Definition 1.36]{Ta} or \cite[Definition 1.4.7.]{Orn}). Indeed, every quasi-equivalence is a pretriangulated equivalence, so it is still true that all the components of both the unit and the counit of the adjunction are pretriangulated equivalences. To prove that $\fdgA$ preserves pretriangulated equivalences it suffices to notice that $\pretr[\ai](\cA)=\pretr[\dg](\cA)$ if $\cA\in\dgCat$ (see \cite[Remark 1.7]{Orn}). As for the fact that $\VdB$ preserve pretriangulated equivalences, suppose that $\fF:\cA\to\cB$ in $\ACat$ induces a quasi-equivalence $\pretr[\ai](\fF)\colon\pretr[\ai](\cA)\to\pretr[\ai](\cB)$. Then
\[
\xymatrix{
\pretr[\ai](\cA)\ar[rrr]^-{\pretr[\ai](\fF)}\ar[d]_-{\pretr[\ai](\ntV[\cA])} & & & \pretr[\ai](\cB)\ar[d]^-{\pretr[\ai](\ntV[\cB])}\\
\pretr[\ai](\VdB(\cA))\ar[rrr]_-{\pretr[\ai](\VdB(\fF))} & & & \pretr[\ai](\VdB(\cB))
}
\]
is a commutative diagram in $\ACat$ in which the upper and the vertical arrows are quasi-equivalences. Hence $\pretr[\ai](\VdB(\fF))=\pretr[\dg](\VdB(\fF))$ is a quasi-equivalence, as well. One can argue in a similar way for Morita equivalences.
\end{remark}

Coming to the proof of \autoref{dgAadj}, this result is proved in \cite[Proposition 2.1]{COS} assuming that $\K$ is a field. Actually the first part of the proof works without changes over an arbitrary commutative ring. Only the argument showing that $\ntV[\cA]\colon\cA\to\VdB(\cA)$ is a quasi-isomorphism for every $\cA\in\ACat$ (respectively, a {\htpiso} when $\cA$ has split units needs to be modified. We now give a different proof, valid over every commutative ring, which occupies the rest of this subsection.

To this aim, we fix a strictly unital $A_\infty$ category $\cA$. As it is explained at the beginning of the proof of \cite[Proposition 2.1]{COS}, $\ntV[\cA]=\nqV[\cA]\comp\ncbi[\cA]$, where the non-unital $A_\infty$ functor $\ncbi[\cA]\colon\cA\to\VdBn(\cA)$ is a {\htpiso} by \autoref{dgAnadj}, and the non-unital dg functor $\nqV[\cA]$ is defined to be the composition
\[
\nqV[\cA]\colon\VdBn(\cA)\mono\aug{\VdBn(\cA)}\epi\aug{\VdBn(\cA)}/\Vid=\VdB(\cA),
\]
with $\Vid=\Vid[\cA]$ the smallest dg ideal of $\aug{\VdBn(\cA)}$ such that $\ntV[\cA]$ is a strictly unital $A_\infty$ functor.

First we recall from \autoref{subsec:themor} (whose notation we adapt and simplify in an obvious way to the setting where $\cB$ is the $0$ dg algebra) that, for every $A,A'\in\cA$,
\[
\VdBn(\cA)(A,A')=\bigoplus_{n\ge0}\grVn{n}(A,A')
\]
as a graded $\K$-module, where
\[
\grVn{n}(A,A'):=\bigoplus_{n_1+\cdots+n_l=n}C_{(n_1,\dots,n_l)}(A,A'),
\]
with
\[
C_{(n_1,\dots,,n_l)}(A,A'):=\bigoplus_{A=A_0,A_1,\dots,A_{l-1},A_l=A'\in\cA}\sh[-1]{C_{n_l}(A_{l-1}A_l)}\otimes\cdots\otimes\sh[-1]{C_{n_1}(A_0,A_1)}
\]
and, for every $i\ge0$,
\[
C_i(A,A'):=\bigoplus_{A=A_0,A_1,\dots,A_{i-1},A_i=A'\in\cA}\sh{\cA(A_{i-1},A_i)}\otimes\cdots\otimes\sh{\cA(A_0,A_1)}.
\]
The differential $d$ on $\VdBn(\cA)(A,A')$ extends $\cod+\com$, where $\cod$ and $\com$ are determined, respectively, by the differential and the comultiplication on the dg cocategory $\Bi(\cA)$. Explicitly, given
\begin{equation}\label{eq:c1}
c=\sh[-1]{(\sh{f_n}\otimes\cdots\otimes\sh{f_1})}\in C_{(n)}(A,A')
\end{equation}
with the $f_i$ homogeneous, we have
\[
\com(c)=\sum_{i=1}^{n-1}(-1)^{\deg(c_{>i})}c_{>i}\otimes c_{\le i}.
\]
The components $\cod^k$ of $\cod$ induced from $\m{k}=\m[\cA]{k}$ are given by 
\begin{equation}\label{eq:cod}
\cod^k(c)=\sum_{i=1}^{n+1-k}\pm\cod^k_i(c),
\end{equation}
(with $1\le k\le n$), where
\[
\cod^k_i(c):=\sh[-1]{\bigl(\sh{f_n}\otimes\cdots\otimes\sh{f_{i+k}}\otimes\sh{\m[\cA]{k}(f_{i+k-1}\otimes\cdots\otimes f_i)}\otimes\sh{f_{i-1}}\otimes\cdots\otimes\sh{f_1}\bigr)}.
\]
As for the signs in \eqref{eq:cod}, we just need to know that they are $(-1)^{\degp(c_{>i})}$ for $k=2$. 

Now we can give the following more explicit description of $\Vid$.

\begin{lem}\label{Vid}
The dg ideal $\Vid$ coincides with the (a priori not necessarily dg) ideal $\Vide$ of $\aug{\VdBn(\cA)}$ generated by all the elements of one of the following two forms:
\begin{enumerate}
\item $1_A-\id_A$, where $A\in\cA$;
\item $c$ as in \eqref{eq:c1} with $n>1$ and such that $f_j=\id_{\tilde{A}}$ for some $j\in\{1,\dots,n\}$ and some $\tilde{A}\in\cA$.
\end{enumerate}
\end{lem}

\begin{proof}
If $c$ is as in \eqref{eq:c1}, we have
\[
\ntV[\cA]^n(f_n\otimes\cdots\otimes f_1)=\nqV[\cA]\bigl(\ncbi[\cA]^n(f_n\otimes\cdots\otimes f_1)\bigr)=\pm\nqV[\cA](c).
\]
Since $\ntV[\cA]$ is strictly unital, it follows that $\nqV[\cA](c)=0$ if $c$ is a generator of $\Vide$ of the form $(2)$. On the other hand, $\nqV[\cA](\id_A)=\id_A$ coincides with the image of $1_A$ through the projection $\aug{\VdBn(\cA)}\epi\VdB(\cA)$, for every $A\in\cA$. Therefore $\Vid$ contains also the generators of $\Vide$ of the form $(1)$, hence $\Vide\subseteq\Vid$. To prove the other inclusion it is clearly enough to show that $d(c)\in\Vide$ for every generator $c$ of $\Vide$. As this is obviously true when $c$ is of the form $(1)$, we can assume that $c$ is of the form $(2)$. Then it is clear from the definition that $\cod^k(c)\in\Vide$ if $k\ne2$, and so it remains to prove that $\Vide$ contains
\[
\cod^2(c)+\com(c)=\sum_{i=1}^{n-1}(-1)^{\degp(c_{>i})}\cod^2_i(c)+\sum_{i=1}^{n-1}(-1)^{\deg(c_{>i})}c_{>i}\otimes c_{\le i}=\sum_{i=1}^{n-1}(-1)^{\degp(c_{>i})}\bigl(\cod^2_i(c)-c_{>i}\otimes c_{\le i}\bigr).
\]
Now, it is immediate to see that (for $0<i<n$) $\cod^2_i(c)\in\Vide$ if $i\ne j,j-1$ and $c_{>i}\otimes c_{\le i}\in\Vide$ if $1<i<n-1$ or $i=1\ne j$ or $i=n-1\ne j-1$. Moreover, if $j=1$ then
\[
\cod^2_1(c)-c_{>1}\otimes c_{\le 1}=c_{>1}\otimes(1_{\tilde{A}}-\id_{\tilde{A}})\in\Vide.
\]
Similarly, if $j=n$ then
\[
\cod^2_{n-1}(c)-c_{\ge n}\otimes c_{<n}=(1_{\tilde{A}}-\id_{\tilde{A}})\otimes c_{<n}\in\Vide.
\]
Finally, if $1<j<n$ then
\[
\cod^2_{j-1}(c)=\cod^2_j(c)=\sh[-1]{(\sh{f_n}\otimes\cdots\sh{f_{j+1}}\otimes\sh{f_{j-1}}\otimes\cdots\otimes\sh{f_1})},
\]
whence $(-1)^{\degp(c_{\ge j})}\cod^2_{j-1}(c)+(-1)^{\degp(c_{>j})}\cod^2_j(c)=0$.
\end{proof}

From \autoref{Vid} we immediately deduce the following result.

\begin{cor}\label{Vnid}
The non-unital dg functor $\nqV[\cA]\colon\VdBn(\cA)\to\VdB(\cA)$ is full and its kernel $\Vnid=\Vnid[\cA]$ is a dg ideal of $\VdBn(\cA)$ such that $\Vnid(A,A')$ (for every $A,A'\in\cA$) is the $\K$-subspace of $\VdBn(\cA)(A,A')$ generated by all the elements of one of the following two forms, where $c^l_{n_l}\otimes\cdots\otimes c^1_{n_1}\in C_{(n_1,\dots,n_l)}(A,A')$:
\begin{enumerate}
\item $c^l_{n_l}\otimes\cdots\otimes c^1_{n_1}-c^l_{n_l}\otimes\cdots\otimes c^{i+1}_{n_{i+1}}\otimes\id_{\tilde{A}}\otimes c^i_{n_i}\otimes\cdots\otimes c^1_{n_1}$ (for suitable $\tilde{A}\in\cA$), with $n_1+\cdots+n_l>0$ and $i\in\{0,\dots,l\}$;
\item $c^l_{n_l}\otimes\cdots\otimes c^1_{n_1}$, with $c^i_{n_i}$ of the form $(2)$ in \autoref{Vid} for some $i\in\{1,\dots,l\}$.
\end{enumerate}
\end{cor}

For every $A,A'\in\cA$ the filtration $\grVn{\le n}(A,A'):=\bigoplus_{m\le n}\grVn{m}(A,A')$ on $\VdBn(\cA)(A,A')$ (where $n\ge0$) induces a filtration $\grVnid{\le n}(A,A'):=\grVn{\le n}(A,A')\cap\Vnid(A,A')$ on $\Vnid(A,A')$ and a filtration
\[
\filt{n}\VdB(\cA)(A,A'):=\bigl(\grVn{\le n}(A,A')+\Vnid(A,A')\bigr)/\Vnid(A,A')\iso\grVn{\le n}(A,A')/\grVnid{\le n}(A,A')
\]
on $\VdB(\cA)(A,A')\iso\VdBn(\cA)(A,A')/\Vnid(A,A')$.

Since $\ncbi[\cA]^1\colon\cA(A,A')\to\grVn{\le1}(A,A')$ is an isomorphism of complexes and $\grVnid{\le1}(A,A')=0$, we see that $\ntV[\cA]^1\colon\cA(A,A')\to\filt{1}\VdB(\cA)(A,A')$ is an isomorphism, as well. Therefore we just need to show that the inclusion $\filt{1}\VdB(\cA)(A,A')\mono\VdB(\cA)(A,A')$ is a quasi-isomorphism, and even a homotopy equivalence if $\cA$ has split units. By \autoref{filthequiv} and \autoref{filtqiso} it is enough to prove that for every $n>1$ the complex $\gr{n}\VdB(\cA)(A,A')$ is null-homotopic, and also that the inclusion $\filt{n-1}\VdB(\cA)(A,A')\mono\filt{n}\VdB(\cA)(A,A')$ splits as a morphism of graded $\K$-modules if $\cA$ has split units.

Now, recall from \autoref{null} and its proof that, for $n>1$, the complex $\grVn{n}(A,A')$ (endowed with the differential $d$ extending $\cod^1+\com$) is null-homotopic, and a map $r\colon\grVn{n}(A,A')\to\grVn{n}(A,A')$ of degree $-1$ satisfying $d\comp r+r\comp d=\id$ can be defined (also for $n=1$) as follows. By linearity an element of $\grVn{n}(A,A')$ can be assumed to be of the form
\begin{equation}\label{eq:c}
c=c^l\otimes\cdots\otimes c^1\in C_{(n_1,\dots,n_l)}(A,A'),
\end{equation}
where $n_1+\cdots+n_l=n$ and $c^k\in C_{(n_k)}(A_{k-1},A_k)$ homogeneous (for $k=1,\dots,l$), with $A_0=A$ and $A_l=A'$. Then
\[
r(c):=
\begin{cases}
0 & \text{if $n_l>1$ or $n=1$} \\
r(c^l\otimes c^{l-1})\otimes c^{l-2}\otimes\cdots\otimes c^1 & \text{if $n_l=1<n$},
\end{cases}
\]
where, if $n_l=1<n$, $c^t=f\in\cA(A_{l-1},A_l)$ and $c^{l-1}=\sh[-1]{(\sh{f_{n_{l-1}}}\otimes\cdots\otimes\sh{f_1})}$, then
\[
r(c^l\otimes c^{l-1}):=(-1)^{\deg(f)}\sh[-1]{(\sh{f}\otimes\sh{f_{n_{l-1}}}\otimes\cdots\otimes\sh{f_1})}.
\]
Since
\[
\gr{n}\VdB(\cA)(A,A')\iso\bigl(\grVn{\le n}(A,A')+\Vnid(A,A')\bigr)/\bigl(\grVn{<n}(A,A')+\Vnid(A,A')\bigr)\iso\grVn{n}(A,A')/\grVnid{n}(A,A'),
\]
where 
\[
\grVnid{n}(A,A'):=\grVn{n}(A,A')\cap\bigl(\grVn{<n}(A,A')+\grVnid{\le n}(A,A')\bigr),
\]
from \autoref{Vnull} we deduce that $\gr{n}\VdB(\cA)(A,A')$ is null-homotopic for $n>1$.

\begin{lem}\label{Vnull}
The map $r\colon\grVn{n}(A,A')\to\grVn{n}(A,A')$ preserves the subcomplex $\grVnid{n}(A,A')$ for every $n>1$ and every $A,A'\in\cA$.
\end{lem}

\begin{proof}
As $r$ preserves both $\grVn{n}(A,A')$ and $\grVn{<n}(A,A')$, it is enough to prove that, if $c\in\grVnid{\le n}(A,A')$, then $r(c)\in\grVn{<n}(A,A')+\grVnid{\le n}(A,A')$. We can clearly assume that $c$ is as in part (1) or (2) of \autoref{Vnid}. In the latter case it is obvious from the definition that $r(c)$ is either $0$ or a generator of the same form in $\grVnid{\le n}(A,A')$. So we can assume $c$ to be of the form $(1)$ with $n_1+\cdots+n_l=n-1$, and it is enough to show that $r(c')\in\grVn{<n}(A,A')+\grVnid{\le n}(A,A')$, where
\[
c':=c^l_{n_l}\otimes\cdots\otimes c^{i+1}_{n_{i+1}}\otimes\id_{\tilde{A}}\otimes c^i_{n_i}\otimes\cdots\otimes c^1_{n_1}.
\]
Now, if $i\ge l-1$, then $r(c')$ is either $0$ or a generator of the form $(2)$ in $\grVnid{\le n}(A,A')$. On the other hand, if $i<l-1$, then $r(c')\in\grVn{<n}(A,A')+\grVnid{\le n}(A,A')$ because $r(c^l_{n_l}\otimes\cdots\otimes c^1_{n_1})\in\grVn{<n}(A,A')$ and $r(c^l_{n_l}\otimes\cdots\otimes c^1_{n_1})-r(c')$ is either $0$ or a generator of the form $(1)$ in $\grVnid{\le n}(A,A')$.
\end{proof}

Finally, \autoref{Vsplit} easily implies that the inclusion $\filt{n-1}\VdB(\cA)(A,A')\mono\filt{n}\VdB(\cA)(A,A')$ splits as a morphism of graded $\K$-modules if $\cA$ has split units and $n>1$.

\begin{lem}\label{Vsplit}
If $\cA$ has split units, then for every $n>1$ and every $A,A'\in\cA$ there exists a morphism of graded $\K$-modules $u\colon\grVn{n}(A,A')\to\grVn{<n}(A,A')$ such that the map
\[
\tilde{u}:=
\begin{pmatrix}
\id & u    
\end{pmatrix}
\colon\grVn{<n}(A,A')\oplus\grVn{n}(A,A')=\grVn{\le n}(A,A')\to\grVn{<n}(A,A')
\]
sends $\grVnid{\le n}(A,A')$ to $\grVnid{<n}(A,A')$.
\end{lem}

\begin{proof}
By hypothesis for every $\tilde{A}\in\cA$ there exists a morphism of graded $\K$-modules $p\colon\cA(\tilde{A},\tilde{A})\to\K$ such that $p(\id_{\tilde{A}})=1$. First, by linearity, every $c\in\grVn{n}(A,A')$ can be assumed to be as in \eqref{eq:c}. Setting 
\[
S(c):=\{i=1,\dots,l\st n_i=1\text{ and }A_{i-1}=A_i\},
\]
we denote, for every subset $S$ of $S(c)$, by $u_S(c)$ the expression obtained from $c$ by deleting the terms $c^i$ with $i\in S$. In case $S=S(c)=\{1,\dots,l\}$ (which implies $A=A'$), we mean $u_S(c)=\id_A$. Now we can define 
\[
u(c):=\sum_{\emptyset\ne S\subseteq S(c)}(-1)^{\card{S}-1}\prod_{i\in S}p(c^i)u_S(c).
\]
It is immediate from the definition that $\tilde{u}$ sends a generator of the form $(2)$ in \autoref{Vnid} to a linear combination of generators of the same form. Hence, given $c$ as above with the additional assumption that there exists $j\in\{1,\dots,l\}$ such that $c^j=\id_{A_j}$ (in particular, $j\in S(c)$), we just need to show that $\tilde{u}(\tilde{c})\in\Vnid(A,A')$, where $\tilde{c}:=u_{\{j\}}(c)-c\in\grVnid{\le n}(A,A')$ is a generator of the form $(1)$. Equivalently, we must prove that $-\tilde{c}+\tilde{u}(\tilde{c})\in\Vnid(A,A')$. In fact we have
\begin{multline*}
-\tilde{c}+\tilde{u}(\tilde{c})=-u_{\{j\}}(c)+c+u_{\{j\}}(c)-u(c)=c-\sum_{\emptyset\ne S\subseteq S(c)}(-1)^{\card{S}-1}\prod_{i\in S}p(c^i)u_S(c) \\
=\sum_{S\subseteq S(c)}(-1)^{\card{S}}\prod_{i\in S}p(c^i)u_S(c)=\sum_{S\subseteq S(c)\setminus\{j\}}(-1)^{\card{S}}\prod_{i\in S}p(c^i)(u_{S\cup\{j\}}(c)-u_S(c)),
\end{multline*}
and each $u_{S\cup\{j\}}(c)-u_S(c)$ is a generator of the form $(1)$ (or $0$ if $S\cup\{j\}=S(c)=\{1,\dots,l\}$).
\end{proof}

This concludes the proof of \autoref{dgAadj}.

\begin{remark}\label{ACatdg}
The proof of the equivalence between $\InfdgCat$ and $\InfACat$ can be modified in an obvious way to prove that there is an equivalence between $\InfdgCat$ and $\Inf\ACatdg$. Hence $\InfACat$ and $\Inf\ACatdg$ are equivalent, as well.
\end{remark}

\subsection{Strictly unital vs.\ unital $A_\infty$ categories}\label{subsec:suu}

A proof of the natural equivalence of $\infty$-categories $\InfACat\to\InfACatu$ has already appeared in \cite{Tan} (see Theorem 1.1 therein). Here we provide a different and simpler proof (the last part of which has already been included in \cite{Tan}). A first simplification consists in the fact that, due to the equivalences $\InfdgCat\to\InfACat$ and $\InfdgCatu\to\InfACatu$, already obtained in \autoref{subsec:Pas}, it is enough to prove that the $\infty$-functor $\Inf\fdgu\colon\InfdgCat\to\InfdgCatu$ is an equivalence of $\infty$-categories. Here $\fdgu\colon\dgCat\to\dgCatu$ denotes the natural inclusion functor, which obviously preserves quasi-equivalences.

First of all, given $\cA\in\dgCatu$ and denoting by $\Id_\cA$ the family of all units in $\cA$, we consider the full dg subcategory $\cC$ of $\pretr(\aug{\cA})$ consisting of the cones (in $\pretr(\aug{\cA})$) of all the morphisms in $\Id_\cA$. Then we define $\fudg(\cA)$ to be the full dg subcategory of the Drinfeld quotient $\pretr(\aug{\cA})/\cC$ whose objects are those in the image of the natural non-unital dg functor
\begin{equation}\label{eq:nudg}
\nudge[\cA]\colon\cA\mono\aug{\cA}\mono\pretr(\aug{\cA})\to\pretr(\aug{\cA})/\cC.
\end{equation}
We will soon see that $\nudge[\cA]$ is actually unital (but not strictly unital). For now it is useful to observe that, while $\cA\mono\aug{\cA}$ is really not unital, the other two maps in \eqref{eq:nudg} are (strictly unital) dg functors. The crucial result is then the following.

\begin{lem}\label{nudgqiso}
The non-unital dg functor $\nudg[\cA]\colon\cA\to\fudg(\cA)$ induced by $\nudge[\cA]$ is a quasi-isomorphism for every $\cA\in\dgCatu$.
\end{lem}

\begin{proof}
Instead of the argument in \cite[Section 2.3]{Tan} we propose the following much shorter one.

Since $\nudge[\cA]$ is injective on objects, $\nudg[\cA]$ is bijective on objects by definition of $\fudg(\cA)$. Thus we just need to show that the $\nudge[\cA]$ is quasi-fully faithful, or, equivalently that the non-unital graded functor
\[
H^*(\nudge[\cA])\colon H^*(\cA)\mono\aug{H^*(\cA)}=H^*(\aug{\cA})\mono H^*\bigl(\pretr(\aug{\cA})\bigr)\to H^*\bigl(\pretr(\aug{\cA})/\cC\bigr)
\]
is fully faithful.

It is an easy exercise (using the fact that the elements of $\Id_\cA$ are units) to prove that the complex of $\K$-modules $\pretr(\aug{\cA})(X,Y)$ is h-flat (even h-projective) for every $X\in\pretr(\aug{\cA})$ and every $Y\in\cC$. Hence \cite[Theorem 3.4]{Dr} shows that the natural functor
\begin{equation}\label{eq:Drinfeld}
H^0\Bigl(\pretr\bigl(\pretr(\aug{\cA})\bigr)\Bigr)/H^0\bigl((\pretr(\cC)\bigr)\to H^0\Bigl(\pretr\bigl(\pretr(\aug{\cA})/\cC\bigr)\Bigr)
\end{equation}
is a triangulated equivalence. Since the natural dg functor $\pretr(\aug{\cA})\mono\pretr\bigl(\pretr(\aug{\cA})\bigr)$ is a quasi-equivalence, the left-hand-side of \eqref{eq:Drinfeld} can be identified with the Verdier quotient $\cT/\cS$, where $\cT:=H^0\bigl(\pretr(\aug{\cA})\bigr)$ and $\cS$ denotes the smallest strictly full triangulated subcategory of $\cT$ containing the objects of $\cC$. Note that $\cS$ is the smallest strictly full triangulated subcategory of $\cT$ containing the cones (in $\cT$) $C_A$ of $\id_A$, for every $A\in\cA$. Here $\id_A$ is the identity of $A$ in $H^0(\cA)$, while $1_A$ will be the identity of $A$ in $H^0(\cA)^+$. It should then be clear that $H^*(\nudge[\cA])$ can be identified with the natural non-unital graded functor
\[
\fnc\colon H^*(\cA)\mono\aug{H^*(\cA)}\mono\grc\cT\to\grc{(\cT/\cS)},
\]
where $\grc\cT$ is the graded category with the same objects as $\cT$ and $\grc{\cT}(X,Y)^n:=\cT(X,\sh[n]{Y})$ for all $X,Y\in\cT$ and all $n\in\ZZ$ (hence $\grc\cT$ can be identified with $H^*\bigl(\pretr(\aug{\cA})\bigr)$), and similarly for $\grc{(\cT/\cS)}$. We have to show that $\fnc$ is fully faithful, and to this purpose, up to passing to the idempotent completion, we can assume $\cT$ to be idempotent complete.

Under this assumption, for every $A\in\cA$, the idempotent morphism $\id_A$ splits in $\cT$. Therefore we can write $A=X_A\oplus Y_A$ for suitable $X_A,Y_A\in\cT$, and $\id_A$ (respectively $1_A-\id_A$) gets identified with $A\epi X_A\mono A$ (respectively $A\epi Y_A\mono A$). Note that this implies $C_A\iso Y_A\oplus\sh[1]{Y_A}$.

On the other hand, using the fact that the composition of two morphisms in $\aug{H^*(\cA)}$ is trivial if one of the two is in $H^*(\cA)$ and the other one is of the form $1_A-\id_A$, it is easy to see that, for every $A,A'\in\cA$, in the decomposition
\[
\aug{H^*(\cA)}(A,A')=\grc\cT(A,A')=\grc\cT(X_A,X_{A'})\oplus\grc\cT(X_A,Y_{A'})\oplus\grc\cT(Y_A,X_{A'})\oplus\grc\cT(Y_A,Y_{A'})
\]
the first summand gets identified with $H^*(\cA)(A,A')$, the second and the third one are always $0$, and the fourth one is a copy of $\K$ in degree $0$ (with generator corresponding to $1_A-\id_A$) if $A=A'$ and $0$ otherwise.

Furthermore $\grc\cT(X_A,X_{A'})\to\grc{(\cT/\cS)}(X_A,X_{A'})$ is an isomorphism, thanks to the implication III) $\Rightarrow$ v) in \cite[Chapt.\ II, Prop. 2.3.3a)]{Ve} (which applies since $\cS$ is orthogonal to all the objects of the form $X_A$). To conclude, it is then enough to note that $\grc{(\cT/\cS)}(A,A')=\grc{(\cT/\cS)}(X_A,X_{A'})$ because $Y_A$ and $Y_{A'}$ are direct summands of objects in $\cS$, and thus they are trivial in $\cT/\cS$. 
\end{proof}

Observe that, as a consequence of \autoref{nudgqiso}, $\nudg[\cA]$ is a unital dg functor (see \autoref{unithtpiso}). It is then very easy to deduce that $\fudg$ extends to a functor $\fudg\colon\dgCatu\to\dgCat$ and that we obtain a natural transformation $\nudg\colon\id_{\dgCatu}\to\fdgu\comp\fudg$, whose components are all quasi-equivalences. Moreover $\fdgu\comp\fudg$, hence also $\fudg$, preserves quasi-equivalences.

On the other hand, since $\nudg[\cA]$ is not strictly unital, $\nudg$ does not define a natural transformation $\id_{\dgCat}\to\fudg\comp\fdgu$. However, we are going to see that there exist another functor $\fudga\colon\dgCat\to\dgCat$ and two natural transformations $\nudga\colon\id_{\dgCat}\to\fudga$ and $\nudgb\colon\fudg\comp\fdgu\to\fudga$, whose components are all quasi-equivalences. By \autoref{Tan} this will allow to conclude that $\Inf\fdgu$ and $\Inf\fudg$ are quasi-inverse equivalences of $\infty$-categories.

The definitions of $\fudga$ and $\nudga$ are similar to those of $\fudg$ and $\nudg$. More precisely, given $\cA\in\dgCat$ and still denoting by $\Id_\cA$ the family of all units in $\cA$, we consider the full dg subcategory $\cC'$ of $\pretr(\cA)$ consisting of the cones (in $\pretr(\cA)$) of all the morphisms in $\Id_\cA$. Then we define $\fudga(\cA)$ to be the full dg subcategory of the Drinfeld quotient $\pretr(\cA)/\cC'$ whose objects are those in the image of the natural (strictly unital) dg functor $\cA\mono\pretr(\cA)\to\pretr(\cA)/\cC'$. Thus we obtain a dg functor $\nudga[\cA]\colon\cA\to\fudga(\cA)$. Using the fact that the objects of $\cC'$ are trivial in cohomology, a much simpler argument than that in the proof of \autoref{nudgqiso} shows that $\nudga[\cA]$ is a quasi-isomorphism. It is then clear that $\fudga$ extends to a functor $\fudga\colon\dgCat\to\dgCat$ and that $\nudga\colon\id_{\dgCat}\to\fudga$ is a natural transformation, whose components are all quasi-equivalences.

Finally, for every $\cA\in\dgCat$, the unique dg functor $\aug{\cA}\to\cA$ such that $\cA\mono\aug{\cA}\to\cA$ is $\id_\cA$ extends to a dg functor $\pretr(\aug{\cA})\to\pretr(\cA)$ mapping $\cC$ to $\cC'$, whence it induces a dg functor $\nudgb[\cA]\colon\fudg(\cA)\to\fudga(\cA)$. It is easy to see, by construction, that $\nudga[\cA]=\nudgb[\cA]\comp\nudg[\cA]$. Since $\nudg[\cA]$ and $\nudga[\cA]$ are quasi-isomorphism, $\nudgb[\cA]$ is a quasi-isomorphism, as well. Clearly we conclude that $\nudgb\colon\fudg\comp\fdgu\to\fudga$ is a natural transformation, whose components are all quasi-equivalences.

\subsection{Unital vs.\  cohomologically unital $A_\infty$ categories}\label{subsec:ucu}

In this subsection we prove the last equivalence in \autoref{infequiv}, namely $\Inf\ACatu\to\Inf\ACatc$. As we already know the natural equivalences $\Inf\dgCatu\to\Inf\ACatu$ and $\Inf\dgCatc\to\Inf\ACatc$ (see \autoref{adjunit}), it is enough to prove that $\Inf\fuc\colon\Inf\dgCatu\to\Inf\dgCatc$ is an equivalence of $\infty$-categories. Here $\fuc\colon\dgCatu\to\dgCatc$ denotes the inclusion functor, which obviously preserves quasi-equivalences.

We start by recalling that, in general, $\cA\in\ACatn$ is \emph{h-projective} if $\cA(A,A')$ is a h-projective complex\footnote{A complex in an abelian category is \emph{h-projective} if every morphism from it to an acyclic complex is null-homotopic.} of $\K$-modules for every $A,A'\in\cA$. We will need the following easy result.

\begin{lem}\label{chproju}
Let $\cA$ be a h-projective and cohomologically unital dg category. Then $\cA$ is unital.
\end{lem}

\begin{proof}
Given a cohomological unit $\unit_A$ of some $A\in\cA$, the difference $\unit_A-\unit_A\comp\unit_A$ is a coboundary. Since composition with a coboundary is always null-homotopic, it follows that both maps
\begin{equation}\label{eq:unitmult}
\farg\comp\unit_A\colon\cA(A,A')\to\cA(A,A') \qquad \unit_A\comp\farg\colon\cA(A',A)\to\cA(A',A)
\end{equation}
are idempotent endomorphisms in the homotopy category of complexes, for every $A'\in\cA$. On the other hand \eqref{eq:unitmult} induce isomorphisms (in fact the identity) in cohomology. As $\cA$ is h-projective, this implies that \eqref{eq:unitmult} are isomorphisms in the homotopy category of complexes. Taking into account that an idempotent isomorphism is necessarily the identity, we conclude that \eqref{eq:unitmult} are homotopic to the identity.
\end{proof}

Then it is enough to show that non-unital dg categories admit functorial h-projective resolutions, meaning that there exist a functor $\fcue\colon\dgCatn\to\dgCatn$ together with a natural transformation $\nuc\colon\fcue\to\id_{\dgCatn}$ such that $\fcue(\cA)$ is h-projective and $\nuc[\cA]\colon\fcue(\cA)\to\cA$ is a quasi-isomorphism, for every $\cA\in\dgCatn$. Indeed, assuming this, by \autoref{chproju} (and recalling \autoref{unithtpiso}) $\fcue$ restricts to a functor $\fcu\colon\dgCatc\to\dgCatu$ and $\nuc$ to natural transformations $\fuc\comp\fcu\to\id_{\dgCatc}$ and $\fcu\comp\fuc\to\id_{\dgCatu}$, whose components are all quasi-equivalences. This also implies that $\fcu$ preserves quasi-equivalences, and we conclude from \autoref{Tan} that $\Inf\fuc$ and $\Inf\fcu$ are quasi-inverse equivalences of $\infty$-categories.

As for the existence of functorial h-projective resolutions for non-unital dg categories, we just very briefly sketch the argument, as one can proceed similarly to the strictly unital case, for which we refer to \cite[Section 3.2]{CNS}. Indeed, in \cite{CNS} the h-projective resolution of $\cA\in\dgCat$ is obtained as the colimit of a sequence of dg categories $\cA_n$ ($n\in\NN$), all with the same objects as $\cA$. In the non-unital case we can keep the same definition of $\cA_1$ and change the inductive step from $\cA_{n-1}$ to $\cA_n$ as follows. While in \cite{CNS} $\cA_n$ is the graded $\K$-linear \emph{category} freely generated over $\cA_{n-1}$ by a suitable set of generators, in the non-unital case we consider the graded $\K$-linear \emph{non-unital category} freely generated over $\cA_{n-1}$ by the same set of generators. In \cite{Orn2} it is explained how the same result holds more generally for (non-unital) $A_\infty$ categories.

\section{$\Hqe$ and categories of functors}\label{sec:1equiv}

In this section, we show that $\Hqe$ is naturally equivalent to the categories $\QACatuhp$ and $\QACathps$, thus proving the second part of \autoref{IntHom}.

To clarify the notation, we set $\ACatuhp$ (respectively $\ACathps$) to be the full subcategory of $\ACatu$ (respectively $\ACat$) whose objects are h-projective (respectively h-projective and with split units, see \autoref{def:splitunit}). Such equivalences are established in \autoref{subsec:equivcounit} (see, in particular, \autoref{unitalhp}), after a good amount of technical preliminaries in \autoref{subsec:prel}.

Actually the effort to prove \autoref{unitalhp} will receive an extra reward in \autoref{sec:IntHom}, where we will use \autoref{hproj} (the key technical result of this section) to prove the first part of \autoref{IntHom}.

\subsection{Preliminary results}\label{subsec:prel}

We begin with the following result about $A_\infty$ functors which corrects a similar statement in \cite{LH}.

\begin{lem}\label{steq}
Let $\fF\colon\cA\to\cB$ be a unital $A_\infty$ functor between two strictly unital $A_\infty$ categories. If $\cA$ has split units, then $\fF$ is homotopic to a strictly unital $A_\infty$ functor.
\end{lem}

\begin{proof}
Since $\fF$ is unital, for every $A\in\cA$ there exists $h_A\in\cB(\fF^0(A),\fF^0(A))^{-1}$ such that $\fF^1(\id_A)=\id_{\fF^0(A)}-\m[\cB]{1}(h_A)$. Then we define a prenatural transformation $\nat\colon\fF\to\fF$ of degree $0$ as follows: for every $f\in\cA(A,A')$ we set
\[
\theta^1(f):=
\begin{cases}
p(f)h_A & \text{if $A=A'$} \\
0 & \text{if $A\ne A'$}
\end{cases}
\]
(where $p\colon\cA(A,A)\to\K$ is as in the proof of \autoref{Vsplit}) and $\nat^i:=0$ for $i\ne1$. By \autoref{nathtp} we can find $\widetilde\fF\in\ACatn(\cA,\cB)$ such that $\fF\htp\widetilde\fF$ and $\widetilde\fF^i=\fF^i+\m{1}(\nat)^i$ for $i>0$. By definition for every $A\in\cA$ we have
\[
\widetilde\fF^1(\id_A)=\fF^1(\id_A)+\m{1}(\nat)^1(\id_A)=\id_{\fF^0(A)}-\m[\cB]{1}(h_A)+\nat^1\bigl(\m[\cA]{1}(\id_A)\bigr)+\m[\cB]{1}\bigl(\nat^1(\id_A)\bigr)=\id_{\fF^0(A)}.
\]
To conclude, using \autoref{htpeqrel} and an easy recursive argument, it should be clear that it is enough to prove the following statement. Assume that $\fF^1(\id_A)=\id_{\fF^0(A)}$ for every $A\in\cA$ and that there exist $n>1$ and $1\le m\le n$ such that $\fF^i(f_i\otimes\cdots\otimes f_1)=0$ if there exists $j\in\{1,\dots,i\}$ such that $f_j=\id_A$ (for some $A\in\cA$) and either $1<i<n$ or $i=n$ and $j<m$. Then we can find $\fG\in\ACatn(\cA,\cB)$ such that $\fF\htp\fG$ through a homotopy $\nat$ with $\nat^i=0$ for $i<n-1$, $\fG^i=\fF^i$ for $i<n$ and $\fG^n(f_n\otimes\cdots\otimes f_1)=0$ if there exists $j\in\{1,\dots,n\}$ such that $f_j=\id_A$ (for some $A\in\cA$) and $j\le m$.

To this aim, a direct but tedious check shows that we can define $\nat$ by
\begin{gather*}
\nat^{n-1}(f_{n-1}\otimes\cdots\otimes f_1):=(-1)^m\fF^n(f_{n-1}\otimes\cdots\otimes f_m\otimes\id_A\otimes f_{m-1}\otimes\cdots\otimes f_1) \\
\nat^n(f_n\otimes\cdots\otimes f_1):=(-1)^m\fF^{n+1}(f_n\otimes\cdots\otimes f_m\otimes\id_A\otimes f_{m-1}\otimes\cdots\otimes f_1)
\end{gather*}
and $\nat^i:=0$ for $i\ne n-1,n$. See also \cite[Lemma 3.7]{Orn} (where the assumption that the category has split units is erroneously missing) for more details of the computation.
\end{proof}

The following is then a straightforward consequence using \autoref{hisocrit} (1).

\begin{cor}\label{cor:split}
If $\cA,\cB\in\ACat$ and $\cA$ has split units, then the natural injective map
\[
\QACat(\cA,\cB)\to\QACatu(\cA,\cB)
\]
is also surjective.
\end{cor}

With an argument similar to the one in the proof of \autoref{steq}, we get the following result about natural transformations.

\begin{lem}\label{striunit}
Let $\fF,\fG\in\ACat(\cA,\cB)$ and let $\nat\colon\fF\to\fG$ be a natural transformation of degree $p$. Then there exists a prenatural transformation $\enat\colon\fF\to\fG$ of degree $p-1$ such that $\nat-\m{1}(\enat)\colon\fF\to\fG$ is a strictly unital natural transformation.
\end{lem}

\begin{proof}
The argument is similar (and a bit simpler) to the one of \autoref{steq}. In this case the only key step consists in the proof of the following statement. Assume that there exist $n>0$ and $1\le m\le n$ such that $\nat^i(f_i\otimes\cdots\otimes f_1)=0$ if there exists $j\in\{1,\dots,i\}$ such that $f_j=\id_A$ (for some $A\in\cA$) and either $0<i<n$ or $i=n$ and $j<m$. Then we can find a prenatural transformation $\inat\colon\fF\to\fG$ of degree $p-1$ such that $\inat^i=0$ for $i<n-1$, $\m{1}(\inat)^i=0$ for $i<n$ and $\m{1}(\inat)^n(f_n\otimes\cdots\otimes f_1)=\nat^n(f_n\otimes\cdots\otimes f_1)$ if there exists $j\in\{1,\dots,n\}$ such that $f_j=\id_A$ (for some $A\in\cA$) and $j\le m$. Here we can define $\inat$ by
\begin{gather*}
\inat^{n-1}(f_{n-1}\otimes\cdots\otimes f_1):=(-1)^m\nat^n(f_{n-1}\otimes\cdots\otimes f_m\otimes\id_A\otimes f_{m-1}\otimes\cdots\otimes f_1) \\
\inat^n(f_n\otimes\cdots\otimes f_1):=(-1)^m\nat^{n+1}(f_n\otimes\cdots\otimes f_m\otimes\id_A\otimes f_{m-1}\otimes\cdots\otimes f_1)
\end{gather*}
and $\inat^i:=0$ for $i\ne n-1,n$. See also \cite[Lemma 3.8]{Orn} for more details.
\end{proof}

We can then prove the following.

\begin{lem}\label{Genovese}
If $\fF,\fF'\in\ACatdg(\cA,\cB)$ are such that $\fF\hiso\fF'$, then $\fF$ and $\fF'$ have the same image in $\HoACatdg$.
\end{lem}

\begin{proof}
This is \cite[Lemma 2.10]{COS}. The only point of the proof that must be modified is the existence of a suitable strictly unital natural transformation $\fF\to\fF'$, for which we invoke \autoref{striunit}.
\end{proof}

In the following we will need to use the fact that $\dgCat$ admits a model structure, where the weak equivalences are the quasi-equivalences and the fibrations are the full dg functors whose $H^0$ is an isofibration (see \cite{Ta}). Recall that, in general, if $\cC$ is a model category, $X\in\cC$ is cofibrant and $Y\in\cC$ is fibrant, then the natural map $\cC(X,Y)\to\Ho{\cC}(X,Y)$ induces a bijection
\begin{equation}\label{eq:htpm}
\xymatrix{
\cC(X,Y)/\htpm \ar@{<->}[r]^-{1:1} & \Ho{\cC}(X,Y)
}
\end{equation}
(see \cite[Theorem 1.2.10]{Hov}), where the equivalence relation $\htpm$ on $\cC(X,Y)$ can be defined as follows.\footnote{Usually this equivalence relation is called homotopy and is denoted by $\sim$, but we will not do that, in order to avoid confusion with the already defined notion of homotopy for $A_\infty$ functors.} First a \emph{cylinder object} for $X$ is given by morphisms $i_0,i_1\colon X\to X'$ and a weak equivalence $s\colon X'\to X$ such that $s\comp i_0=s\comp i_1=\id_X$ and $(i_0,i_1)\colon X\coprod X\to X'$ is a cofibration. Then, given $f_0,f_1\in\cC(X,Y)$, we have $f_0\htpm f_1$ if and only if there exist a cylinder object $(X',i_0,i_1,s)$ for $X$ and $h\in\cC(X',Y)$ such that $f_k=h\comp i_k$, for $k=0,1$ (see \cite[Definition 1.2.4 and Corollary 1.2.6]{Hov}). Moreover, if $f\in\cC(X,Y)$ is a weak equivalence between two fibrant and cofibrant objects, then (always by \cite[Theorem 1.2.10]{Hov}) there exists $g\in\cC(Y,X)$ such that $g\comp f\htpm\id_X$ and $f\comp g\htpm\id_Y$.

\begin{remark}\label{composition}
One can easily see that, by construction, the bijection in \eqref{eq:htpm} is indeed natural with respect to pre and post composition, if one restricts to fibrant and cofibrant objects of $\cC$.
\end{remark}

\begin{remark}\label{cylcof}
If $(X',i_0,i_1,s)$ is a cylinder object for a cofibrant object $X$, then $X'$ is cofibrant, as well: this follows immediately from the fact that cofibrations are stable under composition and pushouts (see \cite[Corollary 1.1.11]{Hov}).
\end{remark}

\begin{remark}\label{cofhp}
It is clear from the definition that every dg category is fibrant. On the other hand, if $\cA\in\dgCat$ is cofibrant, then $\cA$ is also h-projective. Indeed, for every $A,B\in\cA$ the complex of $\K$-modules $\cA(A,B)$ is cofibrant by \cite[Proposition 2.3]{To}, hence h-projective by \cite[Lemma 2.3.8]{Hov}. It follows that every dg category admits a h-projective resolution, namely a quasi-equivalence from a h-projective dg category.
Actually, we can be more precise as, by \cite[Lemma B.5]{Dr}, every $\cA\in\dgCat$ has a semi-free resolution, which is cofibrant by \cite[Lemma B.6]{Dr} (hence h-projective) and has split units.
\end{remark}

\begin{remark}\label{tensorhp}
If $\cA$ and $\cB$ are h-projective dg-categories, then so is $\cA\otimes\cB$.
\end{remark}

\begin{lem}\label{htpmhiso}
If $\fF_0,\fF_1\in\dgCat(\cA,\cB)$ are such that $\cA$ is cofibrant and $\fF_0\htpm\fF_1$, then $\fF_0\hiso\fF_1$.
\end{lem}

\begin{proof}
By definition there exist a cylinder object $(\cA',\fI_0,\fI_1,\fS)$ for $\cA$ and $\fH\in\dgCat(\cA',\cB)$ such that $\fF_k=\fH\comp\fI_k$, for $k=0,1$. Note that both $\cA'$ (by \autoref{cylcof}) and $\cA$ are cofibrant, hence h-projective by \autoref{cofhp}. So the quasi-equivalence $\fS$ is actually a \htpequiv. Since $\fS\comp\fI_0=\fS\comp\fI_1$, part (2) of \autoref{hisocrit} implies $\fI_0\hiso\fI_1$. It follows that $\fF_0=\fH\comp\fI_0\hiso\fH\comp\fI_1=\fF_1$.
\end{proof}

\begin{lem}\label{cof}
Given $\fF\in\ACatdg(\cA,\cB)$ with $\cA$ cofibrant and with split units, there exists $\fF'\in\dgCat(\cA,\cB)$ such that $\fF\hiso\fF'$.
\end{lem}

\begin{proof}
By \autoref{dgAadj} the diagram in $\ACat$
\[
\xymatrix{
\cA \ar[d]_-{\fF} \ar[rr]^-{\ntV[\cA]} & & \VdB(\cA) \ar[d]_-{\VdB(\fF)} \ar[rr]^-{\nfV[\cA]} & & \cA \ar[d]^-{\fF} \\
\cB \ar[rr]_-{\ntV[\cB]} & & \VdB(\cB) \ar[rr]_-{\nfV[\cB]} & & \cB
}
\]
is such that the square on the left commutes. Instead the square on the right commutes when $\fF$ is a dg functor, but not in general. Taking into account that $\nfV[\cA]\comp\ntV[\cA]=\id_\cA$ and $\nfV[\cB]\comp\ntV[\cB]=\id_\cB$, in any case we have
\[
\fF\comp\nfV[\cA]\comp\ntV[\cA]=\fF=\nfV[\cB]\comp\ntV[\cB]\comp\fF=\nfV[\cB]\comp\VdB(\fF)\comp\ntV[\cA].
\]
Since $\ntV[\cA]$ is a \htpequiv, from part (3) of \autoref{hisocrit} we obtain
\begin{equation}\label{eq:hiso1}
\fF\comp\nfV[\cA]\hiso\nfV[\cB]\comp\VdB(\fF).
\end{equation}
Now, let $\fS\colon\cC\to\VdB(\cA)$ be a quasi-equivalence in $\dgCat$ with $\cC$ cofibrant (given, for instance, by a cofibrant replacement of $\VdB(\cA)$). Then $\nfV[\cA]\comp\fS\colon\cC\to\cA$ is a quasi-equivalence in $\dgCat$ between two fibrant and cofibrant objects. Therefore there exists $\fG\in\dgCat(\cA,\cC)$ such that $\nfV[\cA]\comp\fS\comp\fG\htpm\id_\cA$. By \autoref{htpmhiso} this implies
\begin{equation}\label{eq:hiso2}
\nfV[\cA]\comp\fS\comp\fG\hiso\id_\cA.
\end{equation}
Using \eqref{eq:hiso2} and \eqref{eq:hiso1} we obtain
\[
\fF=\fF\comp\id_\cA\hiso\fF\comp\nfV[\cA]\comp\fS\comp\fG\hiso\nfV[\cB]\comp\VdB(\fF)\comp\fS\comp\fG.
\]
To conclude, just observe that $\nfV[\cB]$, $\VdB(\fF)$, $\fS$ and $\fG$ are all dg functors, hence the same is true for their composition.
\end{proof}

\subsection{The missing equivalences}\label{subsec:equivcounit}

The following is the crucial technical result of this section. It will also play an important role in the description of the internal Homs in \autoref{sec:IntHom}.

\begin{prop}\label{hproj}
For every $\cA,\cB\in\dgCat$ with $\cA$ h-projective there is a natural bijection
\[
\xymatrix{
\QACatu(\cA,\cB) \ar@{<->}[r]^-{1:1} & \Hqe(\cA,\cB).
}
\]
\end{prop}

\begin{proof}
First we claim that we can assume that $\cA$ is semi-free.  Indeed, by \autoref{cofhp}), there exists $\cA\iso\tilde{\cA}$ in $\Hqe$, with $\tilde{\cA}$ semi-free. Hence there is a natural bijection
\[
\xymatrix{
\Hqe(\cA,\cB) \ar@{<->}[r]^-{1:1} & \Hqe(\tilde{\cA},\cB).
}
\]
Taking into account that $\tilde{\cA}$ (by \autoref{cofhp}) and $\cA$ are h-projective, \autoref{hequiv} implies $\cA\iso\tilde{\cA}$ in $\QACatu$. Thus there is a natural bijection
\[
\xymatrix{
\QACatu(\cA,\cB) \ar@{<->}[r]^-{1:1} & \QACatu(\tilde{\cA},\cB).
}
\]
It follows that the existence of the required bijection is equivalent to the existence of a natural bijection
\[
\xymatrix{
\QACatu(\tilde{\cA},\cB) \ar@{<->}[r]^-{1:1} & \Hqe(\tilde{\cA},\cB).
}
\]
This proves the claim, and so for the rest of the proof we assume that $\cA$ is semi-free. 

By \autoref{htpmhiso} the inclusion $\dgCat(\cA,\cB)\mono\ACatu(\cA,\cB)$ induces a map
\begin{equation}\label{eq:comp}
\varphi\colon\dgCat(\cA,\cB)/\htpm{}\to\QACatu(\cA,\cB).
\end{equation}
By \eqref{eq:htpm} it is enough to prove that $\varphi$ is bijective. Indeed, given $\fF,\fF'\in\dgCat(\cA,\cB)$ such that $\fF\hiso\fF'$, by \autoref{Genovese} $\fF$ and $\fF'$ have the same image in $\HoACatdg$, hence also in $\Hqe$ (see \autoref{ACatdg}). Therefore $\fF\htpm\fF'$, again by \eqref{eq:htpm}, and this proves that $\varphi$ is injective. Finally, since $\cA$ has split units, \autoref{cor:split} shows that the natural map
\[
\QACat(\cA,\cB)\to\QACatu(\cA,\cB)
\]
is bijective. We conclude that $\varphi$ is surjective by \autoref{cof}.
\end{proof}

We can now apply the previous result to provide the equivalences in the last part of \autoref{IntHom}.

\begin{prop}\label{unitalhp}
The categories $\Hqe$, $\QACatuhp$ and $\QACathps$ are equivalent. In particular, $\Hqe$ and $\QACatu$ are equivalent if $\K$ is a field.
\end{prop}

\begin{proof}
Let $\cC$ be the full subcategory of $\Hqe$ whose objects are h-projective, and let $\cC'$ be the full subcategory of $\QACatuhp$ whose objects are (strictly unital) dg categories. The inclusion $\cC\mono\Hqe$ is clearly an equivalence, and we claim that the same is true for the inclusion $\cC'\mono\QACatuhp$. Indeed, given $\cA\in\ACatuhp$, by \autoref{Yoneda} and \autoref{hequiv} we can find a \htpequiv{} $\fF\colon\rep\cA\to\cA$ with $\rep\cA\in\dgCat$. Denoting by $\fG\colon\cB\to\rep\cA$ a quasi-equivalence with $\cB$ h-projective, we obtain a quasi-equivalence $\fF\comp\fG\colon\cB\to\cA$ with $\cB\in\cC'$. Since both $\cA$ and $\cB$ are h-projective, $\fF\comp\fG$ is actually a \htpequiv, hence its image in $\QACatuhp$ is an isomorphism by \autoref{hequiv}. This proves the claim, and the equivalence between $\Hqe$ and $\QACatuhp$ follows from the fact that, as an easy consequence of \autoref{hproj}, $\cC$ and $\cC'$ are isomorphic categories.

To conclude, it is obviously enough to prove that the inclusion functor
\[
\QACathps\,\to\QACatuhp
\]
is an equivalence. It is clearly faithful and, by \autoref{cor:split}, it is full as well. It is also essentially surjective because, for any $\cA\in\ACatuhp$, we can find $\cB\in\ACathps$ which is isomorphic to $\cA$ in $\QACatuhp$ (for instance, arguing similarly as above, we can take $\cB$ to be a semi-free resolution of $\rep\cA$).
\end{proof}

\begin{remark}\label{suwrong}
The fully faithful inclusion functor $\QACathps\,\to\QACathp$ is not essentially surjective, even when $\K$ is a field (in which case $\QACathp\,=\QACat$). Indeed, the $0$ dg algebra is certainly in $\QACathp$, but it is not in the essential image of the functor, since there is no morphism in $\ACat$ from it to an object of $\ACathps$.
\end{remark}

\section{Internal Homs via $A_\infty$ functors}\label{sec:IntHom}

In this section we complete the proof of \autoref{IntHom} in \autoref{subsec:proofintHoms}, after introducing some preliminary results about $A_\infty$ multifunctors in \autoref{subsec:multifun}.

\subsection{$A_\infty$ multifunctors}\label{subsec:multifun}

Let us briefly recall some constructions which are carefully described in Sections 1.2 and 1.4 in \cite{COS} and which were originally introduced in \cite{BLM}. For this reason we will be concise in the presentation and we will refer to these original sources for more details.

More specifically, given $\cA_1,\ldots,\cA_n,\cA\in\ACatu$, an \emph{$A_\infty$ multifunctor} from $\cA_1,\ldots,\cA_n$ to $\cA$ is a morphism of graded quivers
\[
\fF\colon\red{\aug{\Bi(\cA_1)}\otimes\cdots\otimes\aug{\Bi(\cA_n)}}\to\sh{\cA}
\]
such that the natural extension
\[
\red{\aug{\Bi(\cA_1)}\otimes\cdots\otimes\aug{\Bi(\cA_n)}}\to\Bi(\cA)
\]
of $\fF$ as graded cofunctor commutes with the differentials. Furthermore, an $A_\infty$ multifunctor is \emph{unital} if all its restrictions are unital. The set of all unital $A_\infty$ multifunctors from $\cA_1,\ldots,\cA_n$ to $\cA$ will be denoted by $\ACatu(\cA_1,\ldots,\cA_n,\cA)$. It is also important to know that,  by \cite[Proposition 8.15]{BLM}, there is a unital $A_\infty$ category $\FunAu(\cA_1,\ldots,\cA_n,\cA)$ whose set of objects is $\ACatu(\cA_1,\ldots,\cA_n,\cA)$ (morphisms are suitably defined prenatural transformations). Note that, if $\cA$ is a dg category, then $\FunAu(\cA_1,\ldots,\cA_n,\cA)$ is a dg category as well.

\begin{prop}\label{bifunqe}
For every $\cA_1,\cA_2,\cA_3\in\ACatu$ there is an isomorphism in $\ACatu$
\[
\FunAu(\cA_1,\cA_2,\cA_3)\iso\FunAu(\cA_1,\FunAu(\cA_2,\cA_3)).
\]
\end{prop}

\begin{proof}
It follows from \cite[Proposition 9.18]{BLM} together with \cite[Proposition 4.12]{BLM}.
\end{proof}

In complete analogy with the case of $A_\infty$ functors (see \autoref{subsec:equivfun}), if $\fF_1$ and $\fF_2$ are in $\ACatu(\cA_1,\ldots,\cA_n,\cA)$, we say that $\fF_1$ and $\fF_2$ are \emph{weakly equivalent} (denoted by $\fF_1\hiso\fF_2$) if they are isomorphic in the category $H^0(\FunAu(\cA_1,\ldots,\cA_n,\cA))$.
The relation $\hiso$ is clearly compatible with compositions and then we can define a quotient (multi)category $\QACatu$ with the same objects and whose morphisms are given by
\[
\QACatu(\cA_1,\ldots,\cA_n,\cA):=\ACatu(\cA_1,\ldots,\cA_n,\cA)/\hiso.
\]

\begin{prop}\label{IHMC2}
For every $\cA_1,\cA_2,\cA_3\in\dgCat$ there is a natural bijection
\[
\xymatrix{
\QACatu(\cA_1\otimes\cA_2,\cA_3) \ar@{<->}[r]^-{1:1} & \QACatu(\cA_1,\cA_2,\cA_3).
}
\]
\end{prop}

\begin{proof}
We just sketch the proof, which is essentially the same as that of \cite[Proposition 3.8]{COS}, with a few adjustments at some points. Keeping the same notation of \cite[Section 3]{COS}, $\dgmf{\fN}\in\dgCatn(\coB(\cC),\red{\aug{\cA_1}\otimes\aug{\cA_2}})$ (where $\cC:=\red{\aug{\B(\cA_1)}\otimes\aug{\B(\cA_2)}}$) is actually a {\htpiso} by \autoref{hiso2}. Then, as $\red{\aug{\cA_1}\otimes\aug{\cA_2}}\in\dgCat$ when $\cA_1,\cA_2\in\dgCat$ (see the proof of \cite[Lemma 3.7]{COS}), by \autoref{hequiv} $\dgmf{\fN}$ is unital and there exists $\fH\in\ACatu(\red{\aug{\cA_1}\otimes\aug{\cA_2}},\coB(\cC))$ such that $\fH\comp\dgmf{\fN}\hiso\id_{\coB(\cC)}$. No further change is needed in the rest of the proof, except that in the end we use \autoref{nathiso} instead of \cite[Lemma 1.6]{Sei}.
\end{proof}

\subsection{Proof of the first part of \autoref{IntHom}}\label{subsec:proofintHoms}

Given $\cA_1,\cA_2,\cA_3\in\dgCat$, we want to prove that there are the following natural bijections (where $\cA_i\hp$ denotes a h-projective resolution of $\cA_i$, for $i=1,2$)
\begin{equation}\label{eq:quadratone}
\xymatrix{
\Hqe(\cA_1\lotimes\cA_2,\cA_3) \ar@{<->}[d]^-{1:1}_-{(\mathrm{A})} & & \\
\Hqe(\cA_1\hp\otimes\cA_2\hp,\cA_3) \ar@{<->}[d]^-{1:1}_-{(\mathrm{B})} & & \Hqe(\cA_1,\FunAu(\cA_2\hp,\cA_3)) \ar@{<->}[d]_-{1:1}^-{(\mathrm{F})} \\
\QACatu(\cA_1\hp\otimes\cA_2\hp,\cA_3) \ar@{<->}[d]^-{1:1}_-{(\mathrm{C})} & & \ar@{<->}[d]_-{1:1}^-{(\mathrm{E})} \Hqe(\cA_1\hp,\FunAu(\cA_2\hp,\cA_3)) \\
\QACatu(\cA_1\hp,\cA_2\hp,\cA_3) \ar@{<->}[rr]^-{1:1}_-{(\mathrm{D})} & & \QACatu(\cA_1\hp,\FunAu(\cA_2\hp,\cA_3)),
}
\end{equation}
since this would imply the wanted natural bijection
\[
\xymatrix{
\Hqe(\cA_1\lotimes\cA_2,\cA_3) \ar@{<->}[r]^-{1:1} & \Hqe(\cA_1,\FunAu(\cA_2\hp,\cA_3)).
}
\]
Now, the existence of (A), and (F) follows from the isomorphisms $\cA_1\lotimes\cA_2\iso\cA_1\hp\otimes\cA_2\hp$ and $\cA_1\hp\iso\cA_1$ in $\Hqe$. Taking into account that $\cA_1\hp\otimes\cA_2\hp$ is h-projective by \autoref{tensorhp}, (B) and (E) are due to \autoref{hproj}. Finally, \autoref{IHMC2} implies (C), whereas (D) is a direct consequence of \autoref{bifunqe}.

This clearly implies that in the symmetric monoidal category $\Hqe$ the internal Hom $\IHom(\cA,\cB)$ between two dg categories $\cA$ and $\cB$ is, up to isomorphism in $\Hqe$, the dg category $\FunAu(\cA\hp,\cB)$.

\begin{remark}\label{rmk:compositions}
Let us now show that our new construction of the internal Homs naturally answers To\"en question about the compatibility with compositions in the introduction. Indeed, as $\cB$ and its cofibrant replacements $\tilde{\cB}$ are isomorphic in $\Hqe$, the universal property of the internal Hom yields the isomorphism
\[
\FunAu(\cA\hp,\cB)\iso\FunAu(\cA\hp,\tilde{\cB})
\]
in $\Hqe$. Finally, as we explained in the proof of \autoref{hproj}, $\cA\hp$ and the cofibrant replacement $\tilde{\cA}$ are isomorphic in $\QACatu$, thus we can simply set
\[
\IHom(\cA,\cB):=\FunAu(\tilde{\cA},\tilde{\cB})
\]
yielding a bijection
\begin{equation}\label{isoHom}
\xymatrix{
\Hqe(\cA,\cB)\ar@{<->}[r]^-{1:1} &\Iso(H^0(\IHom(\cA,\cB)))=\QACatu(\tilde{\cA},\tilde{\cB}),
}
\end{equation}
where the last equality is by definition.

As in the proof \autoref{hproj}, the bijection \eqref{isoHom} boils down to the composition of bijections
\[
\xymatrix{
\Hqe(\cA,\cB)\ar@{<->}[r]^-{1:1} &\Hqe(\tilde{\cA},\tilde{\cB})\ar@{<->}[r]^-{1:1} &\dgCat(\tilde{\cA},\tilde{\cB})/\htpm{}\ar@{<->}[r]^-{1:1} &\QACatu(\tilde{\cA},\tilde{\cB}).
}
\]
Now, the first and the last bijections are compatible with pre and post compositions by definition, while the second one is such in view of \autoref{composition}.
\end{remark}

In order to complete the proof of \autoref{IntHom}, we have to show that $\FunAu(\cA\hp,\cB)$ and $\FunA(\cA\hps,\cB)$ (where $\cA\hps$ denotes a h-projective resolution with split units of $\cA$) are isomorphic in $\Hqe$, for any dg categories $\cA$ and $\cB$. To this aim, note that, in general, if $\cC$ is a h-projective dg category with split units, then the fully faithful embedding $\FunA(\cC,\cB)\mono\FunAu(\cC,\cB)$ is indeed a quasi-equivalence. The argument is the same as in \cite[Corollary 2.6]{COS}, where we replace (the erroneous) \cite[Proposition 2.5]{COS} with \autoref{steq}. Thus $\FunAu(\cA\hps,\cB)\iso\FunA(\cA\hps,\cB)$ in $\Hqe$. On the other hand, $\cA\hp$ and $\cA\hps$ are isomorphic in $\Hqe$ as well. Thus the universal property of the internal Homs implies that $\FunAu(\cA\hp,\cB)\iso\FunAu(\cA\hps,\cB)$ in $\Hqe$. This concludes the proof.

\begin{remark}\label{rmk:KK}
It is worth pointing out that \autoref{IntHom} implies Kontsevich--Keller's Claim in the introduction. Indeed, if $\cA$ is a dg category such that $\cA(A,A')$ is cofibrant, for all $A,A'\in\cA$, then $\cA$ is h-projective by \cite[Lemma 2.3.8]{Hov}. Moreover, if $\cA$ has the additional property that the unit map $\K\to\cA(A,A)$ admits a retraction as a morphism of complexes, for all $A\in\cA$, then $\cA$ has split units. Thus $\cA\iso\cA\hps$ in $\Hqe$, whence $\FunA(\cA\hps,\cB)\iso\FunA(\cA,\cB)$ in $\Hqe$.
\end{remark}


\bigskip

{\small\noindent{\bf Acknowledgements.} Part of this work was carried out while the third author was visiting the Laboratoire de Math\'ematiques d'Orsay (Universit\'e Paris-Saclay) and the Simons Laufer Mathematical Sciences Institute (formely MSRI) with the support of the NSF grant DMS-1928930. We are pleased to thank these institutions for the warm hospitality. We are grateful to James Pascaleff, Yong-Geung Oh and Paul Seidel for answering our emails and providing important feedback. We are particularly grateful to Hiro Lee Tanaka for patiently answering our questions and for sharing his preprint \cite{Tan} and to Mohammed Abouzaid for raising a question about cohomologically unital $A_\infty$ categories which led to a major revision of this paper.}


\end{document}